\documentclass[11pt]{article}

\usepackage{amsmath,amsfonts,amssymb,MnSymbol}

\usepackage{xcolor}

\newcommand{\R}{{\mathbb R}}
\newcommand{\N}{{\mathbb N}}

\DeclareMathOperator*{\esssup}{ess\,sup}
\usepackage{bbm}
\renewcommand{\~}{\widetilde}

\def\B{{\cal B}}
\def\E{{\cal E}}
\def\F{{\cal F}}

\def\F{{\cal F}}
\def\V{{\cal V}}
\def\A{{\cal A}}
\def\e{{\cal E}}

\def\ci{\mbox{$C_0^\infty(E)$}}
\def\tt{\mbox{$(T_t)_{t > 0}$}}
\def\tht{\mbox{$(\widehat{T}_t)_{t > 0}$}}

\def\ga{\mbox{$(G_\alpha)_{\alpha > 0}$}}

\def\e{{\varepsilon}}
\def\vt{\widehat{v}_t}

\newtheorem{defn}{{\it DEFINITION}}
\newtheorem{theo}[defn]{{\it THEOREM}}
\newtheorem{lem}[defn]{{\it LEMMA}}
\newtheorem{prop}[defn]{{\it PROPOSITION}}
\newtheorem{cor}[defn]{{\it COROLLARY}}
\newtheorem{rem}[defn]{{\it REMARK}}
\newtheorem{exam}[defn]{{\it EXAMPLE}}
\newenvironment{proof}{{\bf\it Proof }}{{\vskip 0.1cm \hfill$\Box$}}

\begin{document}

\noindent
{\Large\bf Conservativeness criteria for generalized Dirichlet forms}

\bigskip
\noindent
{\bf Minjung Gim},
{\bf Gerald Trutnau}
\\

\noindent
{\small{\bf Abstract.} We develop sufficient analytic conditions for conservativeness of non-sectorial 
perturbations of symmetric Dirichlet forms which can be represented through a carr\'e du champ on a locally compact separable metric space. These form an important subclass of generalized Dirichlet forms which were introduced in \cite{St1}. In case there exists an associated strong Feller process, 
the analytic conditions imply conservativeness, i.e. non-explosion of the associated process in the classical probabilistic sense. 
As an application of our general results on locally compact separable metric state spaces, we consider a generalized Dirichlet form given on a closed or open subset of $\mathbb{R}^d$ which is given as a divergence free first order perturbation of a symmetric energy form. Then using volume growth conditions of the carr\'e du champ and the non-sectorial first order part, we derive an explicit criterion 
for conservativeness. We present several concrete examples which relate our results to previous ones obtained by different authors. 
In particular, we show that conservativeness can hold for a large variance if the anti-symmetric part of the drift is strong enough to compensate it.
This work continues our previous work on transience and recurrence of generalized Dirichlet forms. } \\ \\
\noindent
{Mathematics Subject Classification (2010): primary; 31C25, 47D07, 60J60; secondary: 60H30.}\\

\noindent
{Key words: generalized Dirichlet forms, non-symmetric Dirichlet forms, conservativeness criteria, non-explosion results, Markov semigroups, Diffusion processes.}

\section{Introduction}\label{1}
Conservativeness criteria for $C_0$-semigroups of contractions, non-explosion criteria for Markov processes 
and related problems are important topics both in analysis and probability theory. These were hence studied by many authors under various aspects (see for instance \cite{Da92, Ga59, Gri86, HoJa96, WaMaU, OU15, Sch98, Stu1, St1,Ta89, Ta91, TaTr} and references therein).\\
Here, we take an analytic point of view which fits to the frame of possibly unbounded and discontinuous coefficients.
The method that we use is a refinement of the method developed by Gaffney \cite{Ga59} and also Davies \cite{Da92} and recently localized by Oshima and Uemura \cite{OU15}. In \cite{OU15} a unified method to obtain the conservativeness of a class of Markov processes associated with lower bounded semi-Dirichlet forms, including symmetric diffusion processes, some non-symmetric diffusion processes and jump type Markov processes is presented. We consider a similar unified approach but our interest focuses more on applications to general elliptic diffusions. 
Consequently, our localization procedure of the Davies method is more adapted to the elliptic diffusion case and quite different from the one in \cite{OU15}. \\
The main purpose of this paper is to develop conservativeness criteria for (Markov processes $\mathbb{M}$ or for $C_0$-semigroups of contractions corresponding to) a generalized Dirichlet form which can be expressed as a linear perturbation of a Dirichlet form which can be represented by a carr\'e du champ. Let us briefly explain our technical and conceptional frame. We consider a locally compact separable metric space $E$, a locally finite (i.e. finite on compacts) positive measure $\mu$ with full support 
on $E$ and a generalized Dirichlet form $\E$ that can be decomposed as
\begin{equation}\label{in}
\E(u,v)=\int_E \Gamma(u,v)d\mu+\int_E u Nv d\mu,
\end{equation}
where $\int_E \Gamma(u,v)d\mu$ is a symmetric Dirichlet form on $L^2(E,\mu)$ represented by a carr\'e du champ $\Gamma$ and $(N,D(N))$ is a linear operator on $L^2(E,\mu)$. The precise conditions are formulated in localized form as (H1), (H2) in section \ref{2} below. We further assume (H3), (H4) which are also formulated in section \ref{2}. 
(H3) corresponds to \cite[Assumption (A)]{Stu1} and its consequence \cite[Lemma 1]{Stu1}, i.e. (H3) allows us
to obtain nice cut-off functions (see (\ref{chin})) and to obtain a suitable exhaustive sequence for the state space 
(see (\ref{E4n})). In Remark \ref{rem3} we explain why any symmetric strongly local and regular Dirichlet form satisfying \cite[Assumption (A)]{Stu1} satisfies (H1)-(H3). Since the semigroups that we consider are in general not analytic, we have to impose the denseness condition (H4), where the set $D_0$ that occurs in (H4) is given as in (\ref{D0}). Remark \ref{operatordense} explains more on $D_0$, (H4) and condition (A) that is just used as an auxiliary assumption to perform further calculations (see the sentence right before condition (A)). In Lemma \ref{rmk1}, we include for the reader's convenience a proof to the fact that the conservativeness of the semigroup $(T_t)_{t>0}$ on $L^{\infty}(E,\mu)$ (obtained from the $L^2(E,\mu)$-semigroup associated to $\E$) is equivalent to the $(\widehat {T}_t)_{t>0}$-invariance on $L^1(E,\mu)$. In Lemma \ref{lem4} we derive similarly to \cite{OU15} an equivalent criterion for the $(\widehat {T}_t)_{t>0}$-invariance in localized form. In order to estimate the limit in Lemma \ref{lem4} by the Davies method, we use the functions $\psi_n$ defined in (\ref{psin0}) via the function $\phi$ defined right before display (\ref{psin0}) and then define the "Davies semigroup" in (\ref{psin}). Then in a series of calculations, starting from 
(\ref{ineq1}), and using the key inequality (\ref{key}) which only holds for divergence free perturbations, i.e. because of (\ref{n1}), we obtain our main Theorem \ref{thm6} and its Corollary \ref{g1}. Theorem \ref{thm6} and  Corollary \ref{g1} 
form the core of our paper and will be used to obtain explicit conservativeness criteria in the symmetric, non-symmetric and non-sectorial case. \\
The organization of the following sections \ref{3} and \ref{4} are then as follows.
In section 3, we consider applications of our core results to the symmetric case. Here our results are comparable to \cite{OU15} (see Example \ref{symdf} and Remark \ref{sturm}) and we recover a result of \cite{Stu1} (see Remark \ref{sturm} and also \cite{Gri86} and \cite{Ta91} and references therein) by applying our main Proposition \ref{symexam}. In subsection \ref{exam10}, we consider sectorial perturbations of symmetric Dirichlet forms. Using Corollary \ref{g1}(iii) we are able to reconfirm and hence shorten the proof of a result on conservativeness from \cite[Lemma 5.4]{RST} in subsection \ref{ShTrex}. In subsection \ref{exam11}, we show that Theorem \ref{thm6} (resp. Corollary \ref{g1}) is also applicable to non-symmetric Dirichlet forms with non-symmetric diffusion matrix. The key observation is that the anti-symmetric part of the diffusion matrix becomes a $\mu$-divergence free vector field after integration by parts. The sufficient criteria (\ref{esti}) and (\ref{ahnp}) for conservativeness extend the result of \cite{TaTr} in the sense that we can now consider invariant measures $\mu=\varphi^2 dx$ where $\varphi \nequiv 1$. We also show that we can recover the result of \cite{TaTr} to some extend in case $\varphi\equiv1$ in
subsection \ref{3.3.1} (cf. Remark \ref{tatrrem}). In section \ref{4}, we consider non-sectorial perturbations of symmetric Dirichlet forms on Euclidean space as introduced in \cite{GiTr}. For the convenience of the reader, we explain in concise form the construction of the underlying generalized Dirichlet form from \cite{GiTr}, how the constructed generalized Dirichlet fits into the frame of section \ref{2}, as well as some of its main properties. Subsequently, we apply the conservativeness criterion of section \ref{2} to formulate Corollary \ref{cor13} and to obtain two different explicit examples. Examples \ref{ex16} and \ref{ex1} show that conservativeness can hold for a large variance if the anti-symmetric part of the drift is strong enough to compensate it. Moreover, Example \ref{ex1} indicates that our conservativeness criteria in dimension one can be in some situations sharper than the ones of \cite{Stu1}, but not as sharp as the Feller test is (cf. Remark \ref{rem15}).\\
Let us finally explain our main motivation for this work. 
Conservativeness criteria lead to uniqueness results both at analytic and probabilistic level. Let us discuss both of these. The non-symmetry assumption (or even the lack of sector condition) is here of particular importance, since it leads to a wider class of semigroups and stochastic processes to which the conservativeness criteria can be applied than the restrictive assumption of symmetry. It is pointed out in \cite{St1} that the $(\widehat {T}_t)_{t>0}$-invariance of the underlying measure $\mu$ is related to the $L^1$-uniqueness of the corresponding infinitesimal generator and can be applied to obtain existence of a unique invariant measure. On the other hand  $(\widehat {T}_t)_{t>0}$-invariance is equivalent to the conservativeness of the dual semigroup  $({T}_t)_{t>0}$ (cf. Lemma \ref{rmk1}). Thus conservativeness criteria can be used to obtain $L^1$-uniqueness and existence of unique invariant measures for Markov semigroups. The second important application of the conservativeness criteria that we study is the relation to new non-explosion results 
for solutions to singular SDE which were constructed probabilistically up to an explosion time in \cite{KR}  and \cite{Zh}. 
There it is shown that certain SDE in $\R^d$ with merely $L^p$-integrability conditions on the dispersion and drift coefficients have pathwise unique and strong solutions up to their explosion times, i.e. the random times at which they leave $\R^d$. Thus, if we can construct 
weak solutions to these SDE via (generalized or non-symmetric) Dirichlet form theory, then the analytic conservativeness criteria lead to new non-explosion results for these SDE. We refer the interested reader to the articles \cite{RST}, \cite{ShTr3} where this kind of application has been studied and to subsection \ref{ShTrex} where the results of this article are applied  to  obtain a considerably shorter proof for conservativeness than in \cite[Lemma 5.4]{RST}. For further related work in the context of applications that we are interested in, we refer to the recent work \cite{W15} where non-explosion and existence and uniqueness of invariant measures is investigated.\\

\section{Framework and a general criterion for conservativeness of a generalized Dirichlet form}\label{2}
Let $(E,d)$ be a locally compact separable metric space and let $\mu$ be a locally finite (i.e. finite on compacts) positive measure on its Borel $\sigma$-algebra $\B(E)$. We assume that $\mu$ has full support. The closure of $A\subset E$ will be denoted by $\overline{A}$ and $A^c:=E\setminus A$ stands for the complement of $A$ in $E$. For $1\leq p<\infty$, let $L^p(E,\mu)$ be the space of equivalence classes of $p$-integrable functions with respect to $\mu$ and $L^\infty(E,\mu)$ be the space of $\mu$-essentially bounded functions. We denote the corresponding norms by $\| \cdot \|_{L^p(E,\mu)}$, $p\in [1,\infty]$ and to make notations easier, we do not distinguish at times between equivalence class and representative. The inner product of the Hilbert space $L^2(E,\mu)$ will be denoted by $(\ ,\ )$. The support of a function $u$ on $E$ (=support of $|u|d\mu$) is denoted by supp($u)$. For any set of functions $W$ on $E$, we will denote by $W_0$ the set of functions $u\in W$ which have a compact support in $E$ and by $W_b$ the set of functions in $W$ which are bounded $\mu$-a.e. and let  $W_{loc}$ be the set of measurable functions $u$ such that for any relatively compact open set $V$, there exists $v\in W$ with $u=v$ $\mu$-a.e. on $V$. Let $W_{0,b}:=W_0 \cap W_b$ and define $W_{loc,b}$ by the set of bounded measurable functions $u$ such that $u\in W_{loc}$. Let $C_0(E)$ be the set of continuous functions $u$ such that supp$(u$) is a compact in $E$ and $C_b(E)$ be the set of bounded continuous functions. We say that a statement holds for $n\gg 1$, if there exists some $N\in \N$ such that the statement holds for any $n\geq N$.\\
Let $(\A,\V)$ be a Dirichlet form (not necessarily symmetric) on $L^2(E,\mu)$ in the sense of \cite[I. Definition 4.5]{MR}. So $\V$ is a real Hilbert space with respect to the norm $\|u\|^2_\V:=\A_1(u,u):=\A(u,u)+(u,u)$. Denote the dual space of $\V$ by $\V'$. Assume that there exists a linear operator $(\Lambda,D(\Lambda,L^2(E,\mu)))$ on $L^2(E,\mu)$ which is a  generator of a sub-Markovian $C_0$-semigroup of contractions $(U_t)_{t>0}$ on $L^2(E,\mu)$ that can be restricted to a $C_0$-semigroup on $\V$. Then the conditions (D1) and (D2) in \cite[Chapter I]{St2} are satisfied. In particular, $\Lambda : D(\Lambda,L^2(E,\mu))\cap \V \longrightarrow \V'$ is closable. Let $(\Lambda,\F)$ be the closure of $\Lambda : D(\Lambda,L^2(E,\mu))\cap \V \longrightarrow \V'$. Then $\F$ is a real Hilbert space with corresponding norm
$$
\|u\|^2_\F:=\|u\|^2_\V+\|\Lambda u \|^2_{\V'}.
$$
By \cite[Lemma I.2.4, p.13]{St2} the adjoint semigroup $(\hat U_t)_{t\ge 0}$ of $(U_t)_{t\ge 0}$ can be extended  
to a $C_0$-semigroup on $\V'$ and the corresponding generator $(\hat\Lambda,D(\hat\Lambda,\V'))$ 
is the dual operator of $(\Lambda, D(\Lambda,\V))$. Let $\hat{\F}:=D(\hat\Lambda,\V')\cap {\V}$. 
Then $\hat{\F}$ is a real Hilbert space with corresponding norm 
$$
\|u\|^2_{\hat{\F}} :=\|u\|^2 _{\V} + \|\hat \Lambda u \|_{{\V}'}^2.
$$
Let the form $\E$ be given by

\[ {\E}(u,v):= \left\{ \begin{array}{r@{\quad\quad}l}
 {\A}(u,v)-{}_{\V'}\langle \Lambda u, v \rangle_\V & \mbox{ for}\ u \in{\F},\ v\in {\V} \\ 
            {\A}(u,v)-{}_{\V'}\langle \hat\Lambda v, u \rangle_\V & \mbox{ for}\ u\in{\V},\ v\in\hat
            {\F} \end{array} \right. \] \\
where ${}_{\V'}\langle \cdot, \cdot \rangle_\V$ denotes the dualization between $\V'$ and
$\V$.  Note that ${}_{\V'}\langle \cdot, \cdot \rangle_\V$  coincides with 
$(\cdot,\cdot)$ when restricted to ${L^2(E,\mu)\times\V}$ and that $\E$ is well-defined. $\E$ is called the {\it bilinear form associated with} $(\A,\V)$ and $(\Lambda, D(\Lambda,L^2(E,\mu)))$ (see \cite[I. Definition 2.9]{St2}). Then $\E$ is a generalized Dirichlet form (see \cite[I. Proposition 4.7]{St2}). \\
Let ($G_\alpha)_{\alpha>0}$ and ($\widehat{G}_\alpha)_{\alpha>0}$ on $L^2(E,\mu)$ be associated with $\E$, i.e. ($G_\alpha)_{\alpha>0}$ is the sub-Markovian $C_0$-resolvent of contractions on $L^2(E,\mu)$ satisfying $G_\alpha (L^2(E,\mu))\subset \F$,
$$
\E_\alpha (G_\alpha f, g)=(f,g),\quad f\in L^2(E,\mu),\ g\in \V,
$$
where $\E_\alpha(u,v):=\E(u,v)+\alpha(u,v)$ for $\alpha>0$ and $(\widehat{G}_\alpha)_{\alpha>0}$ is the adjoint $C_0$-resolvent of contractions of $\ga$ (see \cite[I. Proposition 3.6]{St2}). 
\begin{rem}\label{process}
In contrast to the cases of symmetric or non-symmetric Dirichlet forms (which is covered for $\Lambda\equiv 0$ with 
$\F=\V=\hat\F$) it is not known whether for generalized Dirichlet forms regularity or quasi-regularity alone implies the existence of an associated process (cf. \cite[Chapter 7]{FOT} and  \cite[IV.Theorem 3.5]{MR}). In addition to the quasi-regularity the structural assumption D3 is made in \cite[IV.\,2]{St2} in order to derive the existence of  an $m$-tight special standard process 
$\Bbb M=(\Omega,({\cal F}_t)_{t\ge
  0},(Y_t)_{t\ge 0}, (\mathbb{P}_z)_{z\in E_\Delta})$ with lifetime $\zeta$ such that the process resolvent $R_\alpha
  f$ 
of $\Bbb M$ is
  an $\E$-quasi-continuous $\mu$-version of $G_\alpha f$ for all $\alpha  >0$, $f\in
  L^2(E,\mu)_b$ (cf. \cite[IV. Theorem 2.2]{St2} and also \cite[Theorem 4.2]{PeTr} for the construction of an associated Hunt process under the condition D3 together with the strict quasi-regularity assumption).  Condition D3 is e.g. known to hold when $\F$ contains a dense algebra of bounded functions (see \cite[IV. Theorem 2.1]{St2} and also \cite[Proposition 2.1]{PeTr}), thus in particular satisfied for any Dirichlet form in the sense of \cite{FOT} and \cite{MR}. Later on, we will consider examples. The Dirichlet forms in subsections \ref{exam8}, \ref{exam10} and \ref{exam11} are all regular hence associated with a process in the above sense. Moreover, similarly to \cite[Section 3]{St1} one may show that the generalized Dirichlet form constructed in subsection \ref{4.1} is also associated with a process.
\end{rem}
By \cite[I. Proposition 1.5]{MR}, there exists exactly one linear operator ($L,D(L))$ (resp. $(\widehat{L},D(\widehat{L}))$) on $L^2(E,\mu)$ corresponding to $(G_\alpha)_{\alpha>0}$ (resp. $(\widehat{G}_\alpha)_{\alpha>0}$). Then $(\widehat{L},D(\widehat{L}))$ is  the adjoint operator of  ($L,D(L))$. Let $\tt$ and $(\widehat{T}_t)_{t>0}$ be the $C_0$-semigroups of contractions corresponding to $\ga$ and $(\widehat{G}_\alpha)_{\alpha>0}$ respectively. $\tht$ restricted to $L^1(E,\mu)\cap L^2(E,\mu)$ can be extended to a $C_0$-semigroup of contractions on $L^1(E,\mu)$.   This extension will also be denoted by $\tht$. $\tht$ is not necessarily sub-Markovian, however from (H1) on (see below), the sub-Markovianity follows and is hence assumed to hold.\\
Now we shall define the conservativeness of $\tt$. Since $\tt$ is a sub-Markovian $C_0$-semigroup of contractions on $L^2(E,\mu),$ $\tt$ can be extended to a linear operator on $L^\infty(E,\mu)$. In fact, for $f\in L^\infty(E,\mu)$ with $f \geq 0$ $\mu$-a.e., we may set
$$
T_tf:=\lim_{n\to \infty}T_t f_n
$$
where $f_n\in L^2(E,\mu)\cap L^\infty(E,\mu)$ such that $f_n \nearrow f$ $\mu$-a.e. as $n\to \infty$. Since $\tt$ is positivity preserving, the limit is well-defined $\mu$-a.e. and is independent of the choice of approximating sequence. For general $f\in L^\infty(E,\mu)$, considering the decomposition $f=f^+-f^-$ in positive and negative parts, $T_t f$ is well-defined by $T_t f:= T_t f^+ -T_t f^-$. \\
\centerline{}
\begin{defn}\label{conservative} $\tt$ is said to be conservative if
\begin{equation}\label{tt1}
T_t 1=1\ \mu\text{-a.e. for some (and hence any) } t>0.
\end{equation}
\end{defn}
\begin{rem}
Note that if there exists a process associated with the generalized Dirichlet form $\E$, as pointed out in Remark \ref{process}, then the conservativeness of $\tt$ implies that the process is non-explosive, i.e. $\mathbb{P}_x(\zeta=\infty)=1$ for $\mu$-a.e. (actually even $\E$-quasi-every) $x\in E$. Clearly (since $\mu$ is assumed to have full support), if the transition function 
$P_t f(x):=\mathbb{E}_x[f(X_t)]$ (here $\mathbb{E}_x$ denotes the expectation w.r.t. $\mathbb{P}_x$) is strong Feller, i.e. 
$x\mapsto P_t f(x)$ is continuous in $x\in E$ for any $t>0$ and any bounded Borel measurable function on $E$, then it even holds $\mathbb{P}_x(\zeta=\infty)=1$ for every $x\in E$. The latter is for instance the case for the Dirichlet form in Example \ref{ShTrex}, cf. \cite[Proposition 2.9(ii) and Section 5]{RST}.
\end{rem}
\begin{lem}\label{rmk1}
Let $ D$ be an arbitrary dense subset of $L^1(E,\mu)$. Then, $\tt$ is conservative, if and only if for some (and hence any) $t>0$
\begin{equation}\label{tt2}
\int_E \widehat {T}_t f d\mu =\int_E f d\mu\quad \text{for any } f\in D,
\end{equation}
i.e. $\mu$ is $(\widehat {T}_t)_{t>0}$-invariant.
\end{lem}
\begin{proof}
Since the first statement is obvious, we only show that if (\ref{tt2}) (hence equivalently (\ref{tt1})) holds for some $t>0$, then it holds for all $t>0$.
Assume hence that
$$
T_t 1=1\ \mu\text{-a.e. for some } t>0.
$$
Let $(f_n)_{n\geq 1}\subset L^2(E,\mu)\cap L^\infty(E,\mu)$, $0\leq f_n\nearrow 1$ as $n\to \infty$. Then by definition we obtain
$$
\lim_{n\to\infty}T_t f_n=1,\ \mu\text{-a.e.}
$$
Let $t,s>0$. Since $T_{t+s}f_n =T_t(T_s f_n)$, it suffices to show that
$$
\lim_{n\to \infty}T_sf_n=T_s1 =1, \ \mu\text{-a.e.}
$$
for any $0<s<t$. Let $0<s<t$ and suppose that we do not have
$$
\lim_{n\to \infty}T_s f_n= 1,\ \mu\text{-a.e.}
$$
Then there exists a measurable set $A$ with $0<\mu(A)<\infty$ such that
$$
\lim_{n\to \infty}T_sf_n <1\ \text{on }A.
$$
Since $\tht$ is an $L^1(E,\mu)$ contraction,
$$
\int_E \widehat{T}_t 1_A d\mu\leq\int_E \widehat{T}_s 1_A d\mu =\lim_{n\to \infty}\int_E \widehat{T}_s 1_A f_n d\mu= \lim_{n\to \infty} \int_A T_s f_n d\mu< \int_E 1_A d\mu
$$
which leads to the contradiction.
\end{proof}
\centerline{}
Fix $t>0$. From now on until the end of section \ref{2}, we assume:
\begin{itemize}
	\item[(H1)] Let $(V_n)_{n \geq 1}$ be an arbitrary increasing sequence of relatively compact open sets in $E$ such that $\overline{V}_n\subset V_{n+1}$ and $\cup_{n\geq 1}V_n=E$. Then for $p=1$ or $p=2$, there exist sub-Markovian $C_0$-semigroups of contractions $(\widehat{T}^n_t)_{t>0}$, $n\geq 1$ on $L^p(V_n,\mu)$  with generators $(\widehat{L}^n,D(\widehat{L}^n))$, $n\geq 1$, such that for any non-negative $f\in L^1(E,\mu)\cap L^\infty(E,\mu)$,
$$
\widehat{T}^n_t f:=\widehat{T}_t^n (f\cdot 1_{V_n}) \nearrow \widehat{T}_t f\  \mu\text{-a.e. as }n\to \infty.
$$
	\end{itemize}
Next, we aim to give a general criterion for conservativeness in case the generalized Dirichlet form can be represented locally by a linear perturbation of a symmetric strongly local regular Dirichlet form. By the latter, we mean that there exist a symmetric strongly local regular Dirichlet form $(\E^0,D(\E^0))$ on $L^2(E,\mu)$ in the sense of \cite[I.1.1]{FOT}, expressed as
$$
\E^0(u,v)=\int_E \Gamma(u,v)(x)\mu(dx), \text{ for } u,v\in D(\E^0),
$$
where $\Gamma$ is a positive semidefinite symmetric bilinear form on $D(\E^0)$ with values in $L^1(E,\mu)$ (see \cite{BH}) such that for each $n\geq 1$, $D(\widehat{L}^n)_b \subset D(\E^0)_b$ and there exist a linear operator $N:D(N)\longrightarrow L^1(E,\mu)_{loc}$ on $L^2(E,\mu)$ with $D(\E^0)_{b} \subset D(N)$ such that
\begin{equation}\label{euv}
-\int_{V_n} \widehat{L}^nvud\mu=\int_E \Gamma(u,v)d\mu+\int_{V_n} u Nv d\mu
\end{equation}
for any $v\in D(\widehat{L}^n)_b$ and $u\in D(\E^0)$ and $D(N)$ contains $u\cdot 1_{V_n}$ where $u\in D(\E^0)_{loc,b}$. Here the term strongly local means that $\E^0(u,v)=0$ whenever $u$ is a constant on a neighborhood of supp($v$). The linear operator $(N,D(N))$ needs not to be a generator of a $C_0$-semigroup of contractions on $L^2(E,\mu)$, but satisfies 
\begin{equation}\label{loccon}
v \in D(\E^0)_b, v=\text{const. } \mu\text{-a.e. on } B\in \B(E)\Rightarrow  Nv=0 \ \mu\text{-a.e. on }B,
\end{equation}
\begin{equation}\label{n1}
\int_E Nv d\mu= 0\ \text{for any }v\in D(\E^0)_{0,b},
\end{equation}
\begin{equation}\label{n2}
N\phi(v)=\phi'(v)Nv, \ N(uv)=vNu+uNv  \quad\text{for any } u,v\in D(N)_b, \ \phi\in C_b^1(\R).
\end{equation}
Since $D(\E^0)_{0,b}$ forms an algebra of functions (see for instance \cite[I. Corollary 4.15.]{MR}), we obtain form (\ref{n1}) and (\ref{n2}) that 
\[
\int_E u Nu\, d\mu= 0\ \text{for any }u\in D(\E^0)_{0,b}.
\]
From now on until the end of section \ref{2}, we assume that the following condition holds:
\begin{itemize}
\item[(H2)] for each $n\geq 1$, $(\widehat{L}^n,D(\widehat{L}^n))$ can be represented as in (\ref{euv}) and $(N,D(N))$ satisfies (\ref{loccon}), (\ref{n1}) and (\ref{n2})
\end{itemize}
For the convenience of the reader, we recall here some basic properties of strongly local regular Dirichlet forms, which can be represented by a carr\'e du champ. For any $u\in D(\E^0)$, there is a unique finite measure $\mu_{\langle u \rangle}$ on $E$ called the  energy measure of $u$ such that
$$
\int_E d\mu_{\langle u \rangle}=2\E^0(u,u)
$$
and if $u\in D(\E^0)_b$, then we get
$$
\int_E f d\mu_{\langle u \rangle}=2\E^0(u,fu)-\E^0(u^2,f),
$$
for any $f\in C_b(E)\cap D(\E^0)$.
Then $\mu_{\langle u,v \rangle}$, $u,v\in D(\E^0)$, is defined by polarization, i.e.
$$
\mu_{\langle u,v \rangle}:=\frac{1}{2}\left (\mu_{\langle u+v \rangle}-\mu_{\langle u \rangle}-\mu_{\langle v \rangle}\right )
$$
Since $\mu_{\langle u,v \rangle}$ is bilinear in $u,v$ and $\mu_{\langle u \rangle}$ is positive, we obtain for non-negative $f\in C_b(E)\cap D(\E^0)$
$$
\left | \left( \int_E f \mu_{\langle u \rangle} \right)^{1/2} - \left ( \int_E f \mu_{\langle v \rangle} \right )^{1/2} \right | \leq\left( \int_E f d\mu_{\langle u-v\rangle }   \right)^{1/2}.
$$
The energy measure then satisfies for any $u,v \in D(\E^0)$
$$
\int_E f d\mu_{\langle u,v \rangle}=2 \int_E f \Gamma(u,v)d\mu, \quad f \in C_b(E)\cap D(\E^0).
$$
Since $(\E^0,D(\E^0))$ is strongly local, the energy measures $\mu_{\langle u,v \rangle}$, $u,v\in D(\E^0)$, are strongly local and satisfy the Leibniz and the chain rules. In particular, $\mu_{\langle u \rangle}$ can be extended to $u\in D(\E^0)_{loc}$  and $\Gamma(u,v)$ satisfies the Leibniz and Chain rules (see \cite{FOT} and \cite{Stu1}).\\
\centerline{}
We assume from now on until the end of section \ref{2} that
\begin{itemize}
	\item[(H3)]
there exists a non-negative continuous function $\rho$ on $E$ with
$$
\rho \in D(\E^0)_{loc}
$$
such that for $r>0$
$$
E_r:=\{x\in E:  \rho(x)< r\}
$$
is a relatively compact open set in $E$ and $\cup_{r> 0} E_r =E$. Furthermore, there exists a compact subset $K_0$ of $E$ such that
$$
\Gamma (\rho,\rho),\ N(\rho)\in L^\infty_{loc}(K_0^c,\mu).
$$
\end{itemize}
\centerline{}
\begin{rem}\label{rem3}
Let $(\E,\F)$ be a symmetric strongly local and regular Dirichlet form. Then we may define the part Dirichlet forms $(\E^n,\F^n)$ corresponding to an increasing sequence of relatively compact open sets  $(V_n)_{n\geq 1}$ such that $\cup_{n\geq 1}V_n=E$ where $\F^n=\{ u\in \F : \widetilde{u}=0$ q.e. on $V_n^c \}$ and $\widetilde{u}$ is a quasi continuous version of $u\in \F$ (see \cite[Theorem 4.4.5]{FOT}), i.e.  $\widetilde{u}=0$ up to a capacity zero set on $V_n^c$. Denote the associated semigroups of $(\E^n,\F^n)$ by $(T_t^n)_{t>0}$ and the associated linear operators by $(L^n,D(L^n))$ on $L^2(V_n,\mu)$. Then, obviously $(T_t^n)_{t>0}$ and $(L^n,D(L^n))$ are symmetric. We know that
$$
T_t f=\lim_{n \to \infty}T_t^n f \ \mu\text{-a.e.}
$$
for any $f\in L^2(E,\mu)$ where $T_t^n f:= T_t^n(f\cdot 1_{V_n})$. In particular, if $f$ is non-negative, then $T_t^n f \nearrow T_t f$ $\mu$-a.e. as $n\to \infty$. Moreover, as explained before ($\E,\F$) can be represented by a carr\'e du champ. Thus $(\E,\F)$ satisfies (H1) with $p=2$. Furthermore, for $v\in D(L^n)$,
$$
\E(u,v)=(-L^nv,u), \quad\text{for any }u\in \F
$$
which implies that (\ref{euv}) holds. Putting $N\equiv 0$ implies that (H2) holds. Moreover, if the topology induced by the intrinsic metric $d^{int}$ defined by
$$
d^{int}(x,y):=\sup\left \{u(x)-u(y): u\in \F_{loc}\cap C(E),\ \Gamma(u,u)\leq 1 \ \text{on } E \right \}
$$
introduced in \cite{Stu1} is equivalent to the original topology on $E$ and any balls induced by the intrinsic metric are relatively compact open sets, then we may choose $\rho(x)=d^{int}(x,x_0)$ for some fixed $x_0\in E$ (see \cite[Lemma 1]{Stu1}). Hence (H3) holds.
\end{rem}
\centerline{}
By assumption (H3),
\begin{equation}\label{E4n}
V_n:=E_{4n}, n\geq 1,
\end{equation}
are relatively compact open subsets of $E$ with $\bigcup_{n\geq 1} V_n=E$. From now on fix $(V_n)_{n\geq 1}$ as in (\ref{E4n}) and note that (H1) and (H2) hold for this choice of $(V_n)_{n\geq 1}$. For a function $f$ which has compact support, define
\begin{equation}\label{kf}
k_f:=\min\{ m\in \N: \text{supp}(f)\subset E_m\text{ and }K\subset E_m \},
\end{equation}
where $K$ is an arbitrary but fixed compact subset of $E$ containing $K_0$ as in (H3). Let
\begin{equation}\label{D0}
D_0:=\{ f: f\in L^\infty(E,\mu)\cap L^2(E,\mu)_0 \text{ such that } \widehat{T}_s^n f \in D(\widehat{L}^n) \text{ for any }n\geq k_f,\ s\in [0,t] \}.
\end{equation}
In order to perform comfortably our calculations up to the formulation and proof of Theorem \ref{thm6} below, we do need the following auxiliary assumption

\begin{itemize}
\item[(A)] there exists $f\in D_0$ such that supp($f)\neq \emptyset$.
\end{itemize}
\begin{rem}\label{operatordense}
Assumption (A) will be replaced by the stronger (H4) occurring right after the proof of Theorem \ref{thm6} below.
Note that if $(\widehat{T}_t^n)_{t>0}$, $n\geq 1$ are analytic, then $\widehat{T}_s^n f\in D(\widehat{L}^n)$ for $f\in L^1(E,\mu)\cap L^2(E,\mu)$. Thus (A) and (H4) below trivially hold. In the non-sectorial (i.e. non-analytic) case, we can impose the reasonable assumption that
the coefficients of the generators of $(\widehat{T}_t^n)_{t>0}$, $n\ge 1$, are $p$-fold integrable with respect to the measure $\mu$, where $p$ is as in (H1). Then $C_0^{\infty}(E)\subset D_0$ for instance in the case where $E:=\R^d$ and there are no boundary conditions (cf. \ref{ex16} and \ref{ex1}). In particular, (H4)
below is then also automatically satisfied. Similarly, one can easily obtain nice dense subsets of $D_0$ in case of boundary conditions provided the coefficients are not too singular. To keep this exposition reasonably sized and because of the similarity to the case without boundary conditions, we did not include an example.
\end{rem}
\begin{lem}\label{lem4}
Let $D\subset D_0$ be an arbitrary dense subset of $L^1(E,\mu)$. Then $\tt$ is conservative, if and only if there exists a sequence of functions $(\chi_n)_{n\geq 1} \subset L^2(V_n,\mu)$ such that $0\leq \chi_n \nearrow 1$ as $n\to \infty$ and
$$
\lim_{n\to \infty} \left [ \int_0^t \int_E \frac{d}{ds} \widehat{T}_s^n f\cdot \chi_n d\mu ds \right ] =0
$$
for any $f\in D$ and some (and hence all) $t>0$.
\end{lem}
\begin{proof}
Let $f\in D$ and $(\chi_n)_{n\geq 1}$ be as in the statement. Then by (H1)
\begin{equation}\label{efl}
\int_E \left (\widehat{T}_t f-f \right ) d\mu= \lim_{n\to \infty} \int_E ( \widehat{T}_t^n f-f)  \chi_n d\mu=\lim_{n\to \infty} \left [ \int_0^t \int_E \frac{d}{ds} \widehat{T}_s^n f \cdot \chi_n d\mu ds \right]
\end{equation}
for any $f \in D$ and the assertion follows by Lemma \ref{rmk1}.\\
\end{proof}
\centerline{}
Now we are looking for a more explicit criterion for conservativeness of $\tt$.\\
\centerline{}
From now on unless otherwise stated let us fix $f$ as in (A). Let for $n\geq 1$,
\begin{equation}\label{chin}
\chi_n(x):=1\wedge \left(2-\frac{\rho(x)}{2n} \right )^+
\end{equation}
and $\phi: \R^+ \to \R^+$ in $C^1(\R^+)$ be increasing and such that $\phi(0)=0$ and $\phi(r)\nearrow +\infty$ as $r \nearrow +\infty$. Then define for each $n\geq 1$,
\begin{equation}\label{psin0}
\psi_n(x):= (\phi(\rho(x))-\phi(k_f))^+ \wedge (\phi(4n)-\phi(k_f))^+ .
\end{equation}
Note that $(\chi_n)_{n\geq 1}\subset D(\E^0)_{0,b}$ by (H3) and that $\psi_n\in D(\E^0)_{loc}$. The latter can be seen with the help of \cite[p. 190 vi)]{Stu1}. Now we will use the method of Davies, Oshima and Uemura. Let
\begin{equation}\label{psin}
\widehat{T}_s^{\psi_n}f:= e^{\psi_n}\widehat{T}_s^n(f e^{-\psi_n}).
\end{equation}
Then $\widehat{T}_s^{\psi_n}f \in D(\E^0)\cap L^\infty(V_n,\mu)$ with $\widehat{T}_s^{\psi_n}f=0 $ on $V_n^c$ for any $s>0$, and $\widehat{T}_s^{\psi_n}f=e^{\psi_n}\widehat{T}_s^n f$ for any $n\geq 1$ because $\psi_n\equiv 0$ on $E_{k_f}$ for any $n\geq 1$. For $t>0$, let
$$
\widehat{v}_t:=\int_0^t \widehat {T}_s^{\psi_n} f ds.
$$
Let $n\geq k_f$. By Leibniz and chain rules for $\Gamma$ and $N$, (\ref{euv}), (\ref{loccon}), (\ref{n1}) and Fubini, we obtain that
\begin{eqnarray}\label{ineq1}
&& \left | \int_0^t \int_E \frac{d}{ds}\widehat{T}_s^nf \cdot \chi_n d\mu ds\right |
= \left | \int_{V_n} \widehat{L}^n\left (\int_0^t  \widehat{T}_s^n fds\right ) \cdot \chi_n d\mu\right | \nonumber\\
& = &\left | - \int_{V_n}\Gamma\left (\chi_n, e^{-\psi_n}\vt\right) d\mu -\int_{V_n} \chi_n N\left(e^{-\psi_n}\vt \right) d\mu\right |\nonumber  \\
&= &\left |  \int_{V_n} \left (\Gamma (\chi_n,\psi_n) +N(  \chi_n) \right ) e^{-\psi_n} \vt d\mu - \int_{V_n} \Gamma \left(\chi_n, \vt\right ) e^{-\psi_n}d\mu   \right |\nonumber\\
&\leq&\frac{e^{\phi(k_f)-\phi(2n)}}{2n} \left[ \int_{E_{4n}\setminus E_{2n}} \left |\phi'(\rho)\Gamma(\rho,\rho) +N(\rho) \right | \cdot |\vt| d\mu+\left |\int_{E_{4n}\setminus E_{2n}} \Gamma\left (\rho,\vt \right)  d\mu \right | \right ]\nonumber \\
&\leq &  \frac{e^{\phi(k_f)-\phi(2n)}}{2n} \Bigg [\left \{ \left(\int_{E_{4n}\setminus E_{2n}}  \left( \phi'(\rho)\Gamma(\rho,\rho) \right)^2 d\mu \right)^{1/2}+ \left(\int_{E_{4n}\setminus E_{2n}}  \left(N(\rho)\right) ^2 d\mu \right)^{1/2} \right \} \| \vt \|_{L^2(V_n,\mu)}\nonumber \\
&  &\qquad\qquad\qquad \qquad +\left (\int_{E_{4n}\setminus E_{2n}} \Gamma(\rho,\rho)d\mu \right )^{1/2} \left( \int_{V_n} \Gamma\left( \vt,\vt \right)d\mu \right)^{1/2}\Bigg] \nonumber  \\
&\leq  &\frac{e^{\phi(k_f)-\phi(2n)}}{2n} \Bigg [\left(\esssup_{E_{4n}\setminus E_{2n}} \left( \phi'(\rho)\Gamma(\rho,\rho)\right) \mu({E_{4n}\setminus E_{2n}})^{1/2}+ \| N(\rho)\|_{L^2(E_{4n}\setminus E_{2n},\mu)} \right)\| \vt \|_{L^2(V_n,\mu)}\nonumber \\
& &\qquad\qquad\qquad \qquad + {\esssup_{E_{4n}\setminus E_{2n}}  \Gamma(\rho,\rho)}^{1/2} \mu({E_{4n}\setminus E_{2n}})^{1/2} \left( \int_{V_n} \Gamma\left( \vt,\vt \right)d\mu \right)^{1/2} \Bigg] \nonumber  \\
&\leq & \frac{e^{\phi(k_f)-\phi(2n)} }{2n} \Bigg ( \mu(E_{4n}\setminus E_{2n})^{1/2}\left( \sqrt{a_n}\E^0 \left( \vt,\vt\right)^{1/2}+ b_n\| \vt \|_{L^2(V_n,\mu)} \right ) \nonumber \\
& &\qquad\qquad\qquad \qquad \qquad\qquad\qquad \qquad\qquad\qquad+{\| N(\rho)\|_{L^2(E_{4n}\setminus E_{2n},\mu)}} \| \vt \|_{L^2(V_n,\mu)} \Bigg)
\end{eqnarray}
where
\begin{equation}\label{an}
a_n:= \esssup_{E_{4n}\setminus E_{2n}} \Gamma(\rho,\rho)
\end{equation}
and
\begin{equation}\label{bn}
b_n:=\esssup_{E_{4n}\setminus E_{2n}}\phi'(\rho) \Gamma(\rho,\rho).
\end{equation}
Since $\Gamma$ is positive semidefinite and $\phi$ is increasing, $a_n$ and $b_n$ are nonnegative and well-defined by (H3) and (\ref{kf}). Now, we are going to find the following estimates in (\ref{ineq1})
$$
\left \| \vt \right\|_{L^2(V_n,\mu)} \leq t e^{c_n(f)t}\|f \|_{L^2(E,\mu)}
$$
and
$$
\E^0\left( \vt,\vt \right)^{1/2}\leq \sqrt{3t} e^{c_n(f) t} \|f \|_{L^2(E,\mu)}
$$
where
\begin{equation}\label{cn}
c_n(f):=\esssup_{E_{4n}\setminus E_{k_f}} \left |(\phi'(\rho))^2 \Gamma(\rho,\rho)+\phi'(\rho)N(\rho)\right |.
\end{equation}
Note that  $c_n(f)$ is well-defined by (H3) and (\ref{kf}) and depends on $f$ since the essential supremum is taken over $E_{4n}\setminus E_{k_f}$. Since $N$ satisfies (\ref{n1}) and (\ref{n2}), we obtain the following lemma which is the key lemma of this section. \\
\begin{lem}
Let $V$ be a relatively compact open set in $E$, $u\in D(\E^0)_{0,b}$ with supp$(u)\subset \overline{V}$ and $\psi \in D(\E^0)_{loc,b}$.  Then $e^{\pm \psi}u\in D(\E^0)_b\subset D(N)_b$ and
\begin{equation}\label{key}
\E^0(e^{\psi}u,e^{-\psi}u)+\int_V e^{\psi}u \cdot N(e^{-\psi}u) d\mu\geq \E^0(u,u)-c\int_V u^2 d\mu,
\end{equation}
where
$$
c:=\esssup_{V} \left |\Gamma (\psi,  \psi )+   N(\psi)  \right |.
$$
\end{lem}
\begin{proof} $e^{\pm \psi} u\in D(\E^0)_b$ follows since $(e^{\pm \psi}-1)u\in D(\E^0)_b$. Since $(N,D(N))$ satisfies (\ref{n1}) and (\ref{n2}) and $\Gamma$ satisfies the Leibniz rule,
\begin{eqnarray*}
\E^0(e^{\psi}u,e^{-\psi}u)+\int_V e^{\psi}u \cdot N(e^{-\psi}u) d\mu
&=&\E^0(u,u)-\int_V \left (\Gamma(\psi,\psi)+N(\psi) \right) u^2 d\mu\\
&\geq& \E^0(u,u)-c \int_V u^2 d\mu.
\end{eqnarray*}
\end{proof}
\centerline{}
For $s>0$, we have by (\ref{euv})
\begin{align*}
\frac{1}{2}\frac{d}{ds}\| \widehat{T}_s^{\psi_n}f \|_{L^2(V_n,\mu)}^2
&=\int \widehat{L}^n( \widehat{T}_s^n f )\cdot  e^{\psi_n} \widehat{T}_s^{\psi_n}f d\mu\\
&=-\E^0\left(e^{\psi_n} \widehat{T}_s^{\psi_n} f, e^{-\psi_n} \widehat{T}_s^{\psi_n}f\right)-\int_{V_n} e^{\psi_n} \widehat{T}_s^{\psi_n} f \cdot N(e^{-\psi_n} \widehat{T}_s^{\psi_n} f)d\mu.
\end{align*}
Replacing $u$ by $\widehat{T}_s^{\psi_n}f$, $s>0$ and $\psi$ by $\psi_n$ in (\ref{key}), we obtain
$$
\frac{1}{2}\frac{d}{ds}\| \widehat{T}_s^{\psi_n}f \|_{L^2(V_n,\mu)}^2\leq-\E^0\left( \widehat{T}_s^{\psi_n}f,\widehat{T}_s^{\psi_n}f \right)+c_n(f) \int_{V_n} \left ( \widehat{T}_s^{\psi_n}f \right)^2 d\mu.
$$
Consequently, $\frac{d}{ds}\| \widehat{T}_s^{\psi_n}f \|_{L^2(V_n,\mu)}^2  \leq 2c_n(f) \| \widehat{T}_s^{\psi_n}f \|_{L^2(V_n,\mu)}^2$, i.e.
$$
\left \| \widehat{T}_s^{\psi_n}f \right\|_{L^2(V_n,\mu)} \leq  e^{c_n(f)s}\|f \|_{L^2(E,\mu)}.
$$
By Fubini and Jensen,
\begin{align*}
\left \| \vt \right \|_{L^2(V_n,\mu)}^2
&\leq t \int_0^t \int \left(\widehat{T}_s^{\psi_n}f \right )^2d\mu ds \leq t \int_0^t e^{2c_n(f) s}ds \|f\|_{L^2(E,\mu)}^2
\end{align*}
Hence,
\begin{equation}\label{ineq2}
\| \vt \|_{L^2(V_n,\mu)}^2\leq \frac{t}{2c_n(f)}\left ( e^{2c_n(f) t}-1 \right)  \|f\|_{L^2(E,\mu)}^2\ \text{ and }\ \| \vt \|_{L^2(V_n,\mu)}^2\leq t^2 e^{2c_n(f) t}  \|f\|_{L^2(E,\mu)}^2.
\end{equation}
Next, using (\ref{key}) again we obtain
\begin{align*}
\left( \vt, \widehat{T}_t^{\psi_n}f \right ) -\left(\vt , f\right)
&= \int \int_0^t \frac{d}{du}\widehat{T}_u^{n}f du \cdot e^{\psi_n}\int_0^t \widehat {T}_s^{\psi_n} f ds d\mu\\
&= \int  \widehat{L}^n \left( \int_0^t \widehat{T}_u^{n}f du\right ) \cdot e^{\psi_n}\int_0^t \widehat {T}_s^{\psi_n} f ds d\mu\\
&=-\E^0 \left(e^{\psi_n} \int_0^t \widehat {T}_s^{\psi_n} f ds , e^{-\psi_n} \int_0^t \widehat{T}_u^{\psi_n} f du\right )\\
&\qquad \qquad \qquad -\int_{V_n} e^{\psi_n}\int_0^t \widehat {T}_s^{\psi_n} f ds\cdot N\left ( e^{-\psi_n} \int_0^t \widehat{T}_u^{\psi_n} f du \right)d\mu\\
&\leq -\E^0 \left( \vt, \vt\right )+c_n(f) \left\| \vt \right\|_{L^2(V_n,\mu)}^2.
\end{align*}
Thus, we get by (\ref{ineq2})
\begin{align}\label{ineq3}
\E^0\left( \vt, \vt \right)&\leq c_n(f) \left\| \vt \right\|_{L^2(V_n,\mu)}^2 + \left( \vt,f \right)-\left(\vt,\widehat T_t^{\psi_n}f\right) \nonumber \\
&\leq c_n(f) \left\| \vt \right\|_{L^2(V_n,\mu)}^2+ \left\| \vt \right\|_{L^2(V_n,\mu)}  \|f\|_{L^2(E,\mu)} + \left\| \vt \right\|_{L^2(V_n,\mu)}\left\| \widehat{T}_t^{\psi_n}f \right\|_{L^2(V_n,\mu)}\nonumber \\
&\leq \frac{t}{2}\left ( e^{2c_n(f) t}-1 \right) \|f \|_{L^2(E,\mu)}^2 + t e^{c_n(f) t} \|f\|_{L^2(E,\mu)}^2+t e^{2c_n(f) t} \|f\|_{L^2(E,\mu)}^2\nonumber \\
&\leq 3te^{2c_n(f) t} \|f \|_{L^2(E,\mu)}^2.
\end{align}
Consequently, using the estimates (\ref{ineq2}) and (\ref{ineq3}) in (\ref{ineq1}), we get
\begin{eqnarray}\label{ineq4}
&&\left |\int_0^t \int \frac{d}{ds}\widehat{T}_s^n f \cdot \chi_n d\mu ds\right |  \\ \nonumber
&\leq& \frac{e^{\phi(k_f)-\phi(2n)+c_n(f)t} }{2n} \|f\|_{L^2(E,\mu)}\left ( (\sqrt{3t a_n }+b_n t)\mu(E_{4n}\setminus E_{2n})^{1/2} +t \| N(\rho)\|_{L^2(E_{4n}\setminus E_{2n},\mu)}  \right).
\end{eqnarray}
Let
\begin{equation}\label{ahn}
\widehat{A}_n(\phi):= (\sqrt{a_n}+b_n)\mu(E_{4n}\setminus E_{2n})^{1/2} + \| N(\rho)\|_{L^2(E_{4n}\setminus E_{2n},\mu)}
\end{equation}
where $a_n$ and $b_n$ are defined as in (\ref{an}), (\ref{bn}) respectively. Note that $\widehat{A}_n(\phi)$ depends on the choice of $\phi$ but does not depend on $f$. Lemma \ref{lem4} now leads to the following theorem.
\centerline{}
\begin{theo}\label{thm6}$ $
\begin{itemize}
\item[(i)] Let $f$ be as in (A) and suppose that there exists a continuously differentiable function $\phi: \R^+ \to \R^+$ with $\phi(0)=0$ and $\phi(r)\nearrow +\infty$ as $r \nearrow +\infty$, such that for some constant $T>0$
\begin{equation}\label{lim6}
\limsup_{n\to \infty}\frac{e^{-\phi(2n)+c_n(f) T}}{n}\widehat{A}_n(\phi)=0
\end{equation}
where $\widehat{A}_n(\phi)$ is defined as in (\ref{ahn}). Then
$$
\int_E \widehat{T}_t f d\mu=\int_E f d\mu.
$$
\item[(ii)] Assume that (\ref{lim6}) holds for at least one triple $(f,\phi,T)$ as in (i). Then (\ref{lim6}) holds for the triple $(g,\phi,T)$, for any $g\in D_0$ (see (\ref{D0}) for the definition of $D_0$). In particular, if additionally $D_0$ is dense in $L^1(E,\mu)$, then $\tt$ is conservative.
\end{itemize}
\end{theo}
\begin{proof}
(i) is a direct consequence of (\ref{efl}), (\ref{ineq4}) and (\ref{lim6}). We now prove (ii). Let $(f,\phi,T)$ be as in (i) and $g\in D_0$. It suffices to show that
$$
\limsup_{n\to \infty}\frac{e^{-\phi(2n)+c_n(g) T}}{n}\widehat{A}_n(\phi)=0
$$
where
$$
c_n(g)=\esssup_{E_{4n}\setminus E_{k_g}} \left |(\phi'(\rho))^2 \Gamma(\rho,\rho)+\phi'(\rho)N(\rho)\right |.
$$
If $k_g\geq k_f$, then $E_{k_f}\subset E_{k_g}$ and so $c_n(g)\leq c_n(f)$. Thus (\ref{lim6}) for $(f,\phi,T)$ implies (\ref{lim6}) for $(g,\phi,T)$. If $k_g< k_f$, then
$$
c_n(g)\leq c_n(f) +\esssup_{E_{k_f}\setminus E_{k_g}} \left |(\phi'(\rho))^2 \Gamma(\rho,\rho)+\phi'(\rho)N(\rho)\right |\leq c_n(f)+L
$$
for some constant $L\geq 0$, since $\esssup_{E_{k_f}\setminus E_{k_g}} \left |(\phi'(\rho))^2 \Gamma(\rho,\rho)+\phi'(\rho)N(\rho)\right |$ is finite by (H3) and (\ref{kf}). Thus (\ref{lim6}) holding for the triple $(f,\phi,T)$ again implies (\ref{lim6}) for the triple $(g,\phi,T)$. If additionally $D_0$ is dense, then $\tt$ is conservative by Lemma \ref{lem4}.
\end{proof}
\centerline{}
We formulate the condition of Theorem \ref{thm6}(ii) as
\begin{itemize}
	\item[(H4)] $D_0$ is dense in $L^1(E,\mu)$.
	\end{itemize}
It is clear that (H4) implies (A). Now, we use Theorem \ref{thm6} to develop the following explicit sufficient conditions for conservativeness of $\tt$.
\centerline{}
\begin{cor}\label{g1} Assume that (H1)-(H4) hold.
\begin{itemize}
\item[(i)] Suppose there are constants $M,C>0$, $0<\alpha<1$ and $0\leq \beta<2$, such that
\begin{equation}\label{g11}
 \left |  \Gamma(\rho,\rho) + \frac{(\rho+1)N(\rho)}{C(2-\beta)(\log(\rho+1))^{1-\beta}}  \right |\leq M(\rho+1)^2 (\log(\rho+1))^\beta,
\end{equation}
$\mu$-a.e. outside some arbitrary compact subset $K$ of $E$ with $K\supset K_0$ and
$$
\widehat{A}_n(\phi) \leq n \exp(\alpha C \left (\log(n+1)\right )^{2-\beta}),
$$
for $n\gg 1$, where $ \phi(r)=C(\log(r+1))^{2-\beta}$. Then $\tt$ is conservative.
\item[(ii)]
Suppose there are constants $M,C>0$ and $0<\alpha<1$, such that
$$
 \left |  \Gamma(\rho,\rho) + \frac{1}{C}{(\rho+1)(\log(\rho+1))N(\rho)}  \right |\leq M(\rho+1)^2 (\log(\rho+1))^2,
$$
$\mu$-a.e. outside some arbitrary compact subset $K$ of $E$ with $K\supset K_0$ and
$$
\widehat{A}_n(\phi) \leq n\log(n+1)^{C\alpha},
$$
for $n\gg 1$, where $ \phi(r)=C\log(\log(r+1)+1)$. Then $\tt$ is conservative.
\item[(iii)] Suppose that there are constants $M,C>0$ and $0<\alpha<2$ such that
$$
\left | \Gamma(\rho,\rho)+ \frac{ N( \rho )}{C\rho} \right | \leq M
$$
$\mu$-a.e. outside some arbitrary compact subset $K$ of $E$ with $K\supset K_0$ and
$$
\widehat{A}_n(\phi)\leq n\exp(\alpha C n^2)
$$
for $n\gg 1$, where $\phi(r)=\frac{Cr^2}{2}$. Then $\tt$ is conservative.
\end{itemize}
\end{cor}
\begin{proof} (i) Assume there are constants $M,C>0$, $0<\alpha<1$ and $0\leq \beta<2$ such that (\ref{g11}) holds. Let
$$
\phi(r):=C(\log(r+1))^{2-\beta}.
$$
Since $0\leq \beta <2$, $\phi(r)$ is increasing in $r>0$ and
$$
\phi'(r)=\frac{C(2-\beta)}{(r+1)}(\log(r+1))^{1-\beta}.
$$
By (H4), we can choose  $g\in D_0$ with supp($g)\neq \emptyset$. By definition of $k_g$, we know $K_0\subset K\subset E_{k_g}$. Hence by (\ref{g11}), we obtain that
$$
c_n(g)\leq\esssup_{E_{4n}\setminus K} \left |(\phi'(\rho))^2 \right |\cdot \left |\Gamma(\rho,\rho)+\frac{N(\rho)}{\phi'(\rho)}\right |\leq M'(\log(4n+1))^{2-\beta}
$$
where $M'>0$ is some constant depending only on $M, C$ and $\beta$. Subsequently, for $n\geq k_g$
\begin{eqnarray*}
&&\frac{\widehat {A}_n(\phi)}{n}\exp(-\phi(2n)+c_n(g) T) \\
&\leq & \exp\left (\alpha C(\log(n+1))^{2-\beta}-C(\log(2n+1))^{2-\beta} +M'T(\log(4n+1))^{2-\beta}\right ).
\end{eqnarray*}
Let $T:=\frac{C(1-\alpha)}{2M'}>0$. Then the right hand side of the above inequality tends to 0 as $n\to \infty$ and so (\ref{lim6}) of Theorem \ref{thm6}(i) holds for the triple $(g,\phi,T)$. Using (H4), Theorem \ref{thm6}(ii) applies, i.e. $\tt$ is conservative.\\
(ii) Let $\beta=2$. Putting
$$
\phi(r):=C\log( \log(r+1)+1),
$$
we can proceed as in (i) to show that $\tt$ is conservative.\\
(iii) Let $g\in D_0$ with supp($g)\neq \emptyset$. For $n\geq k_g$,
$$
c_n(g)\leq \esssup_{E_{4n}\setminus K_0} |(\phi'(\rho)^2| \cdot \left | \Gamma(\rho,\rho)+ \frac{N( \rho )}{\phi'(\rho)}\right |\leq \esssup_{E_{4n}\setminus K_0}MC^2\rho^2=16MC^2 n^2,
$$
and so
$$
\frac{e^{-\phi(2n)+c_n(g)T}}{n}\widehat{A}_n(\phi)\leq 
\exp(\alpha Cn^2-2Cn^2+16MC^2Tn^2).
$$
Let $T:=\frac{2-\alpha}{32MC}>0$, then
$$
\lim_{n\to \infty}\frac{e^{-\phi(2n)+c_n(g)T}}{n}\widehat{A}_n(\phi)=0.
$$
Applying Theorem \ref{thm6}(ii), we obtain that $\tt$ is conservative.
\end{proof}
\section{Applications to symmetric and non-symmetric Dirichlet forms}\label{3}
In the fist subsection, we apply Theorem \ref{thm6} to symmetric Dirichlet forms. The results turn out to be comparable with the results of \cite[Section 3.1]{OU15} (cf. Example \ref{symdf} and Remark \ref{sturm} below).
\subsection{Symmetric Dirichlet forms}\label{exam8}
Let $(\E,\F)$ be a symmetric strongly local regular Dirichlet form on $L^2(E,\mu)$ expressed as
\begin{equation}\label{form}
\E(f,g)=\int_E \Gamma(f,g)(x)\mu(dx), \text{ for } f,g\in \F.
\end{equation}
Let us fix an arbitrary $x_0\in E$ and denote $d(x,x_0)$ by $d(x)$ for simplicity. Assume
\begin{equation}\label{dloc}
d\in \F_{loc}
\end{equation}
and that
\begin{equation}\label{relcpt}
E_r:=\{ x\in E: d(x)<r\}\text{ are relatively compact open sets in }E\text{ for any }r>0.
\end{equation}
Assume further that there exists a compact subset $K_0$ of $E$ such that
\begin{equation}\label{sym1}
\Gamma(d,d)\in L^\infty_{loc}(K_0^c,\mu).
\end{equation}
 As we have seen in Remark \ref{rem3}, (H1) and (H2) hold with $p=2$ and  $N\equiv0$. Furthermore, putting $\rho(x)=d(x)$, (H3) also holds by (\ref{dloc}), (\ref{relcpt}) and (\ref{sym1}). Since the semigroups $(T_t^n)_{t>0}$, $n\geq 1$ of the part forms $(\E^n,\F^n)$ on $L^2(V_n,\mu)$ are analytic so that in particular $T_t^nf \in D(L^n)$ for any $f\in L^2(E,\mu)$ and $t>0$, (H4) also holds (obviously $D_0=L^\infty(E,\mu)\cap L^2(E,\mu)_0$ is dense in $L^1(E,\mu)$). Thus we can use Theorem \ref{thm6} to determine conservativeness of the symmetric Dirichlet form $(\E,\F)$. More precisely, we have:
\begin{prop}\label{symexam}$ $
\begin{itemize}
\item[(i)] Assume there are constants $M,N>0$ and $0\leq \beta\leq2$, such that
\begin{equation}\label{first}
 \Gamma(d,d)\leq M(d+1)^2 \left (\log(d+1)\right )^\beta,
\end{equation}
$\mu$-a.e. outside some arbitrary compact subset $K$ of $E$ with $K\supset K_0$ and
$$
\mu(E_{4n}\setminus E_{2n}) \leq \exp(2N \left (\log(n+1)\right )^{2-\beta}), \ \text{if }0\leq \beta <2,
$$
or
$$
\mu(E_{4n}\setminus E_{2n}) \leq \log(n+1)^{2N}, \ \text{if } \beta=2
$$
for $n\gg 1$. Then $\tt$ is conservative.
\item[(ii)] Assume there are constants $M,N>0$ such that
$$
\Gamma(d,d)\leq M
$$
$\mu$-a.e. outside some arbitrary compact subset $K$ of $E$ with $K\supset K_0$ and
$$
\mu(E_{4n}\setminus E_{2n})\leq \exp(2Nn^2)
$$
for $n \gg 1$. Then $\tt$ is conservative.
\end{itemize}
\end{prop}
\begin{proof}
(i) Let $0\leq \beta <2$ and define for $r>0$,
$$
\phi(r):=C \left (\log(r+1)\right )^{2-\beta}
$$
where $C>0$ will be chosen later. Then, $\phi(r)$ is increasing in $r>0$ and
$$
\phi'(r)=\frac{C(2-\beta)}{(r+1)}\left (\log(r+1)\right )^{1-\beta}.
$$
Choose  $g\in D_0$ with supp($g)\neq \emptyset$. For $n\geq k_g$, we have by (\ref{first})
$$
a_n=\esssup_{E_{4n}\setminus E_{2n}} \Gamma(d,d)\leq M(4n+1)^2(\log(4n+1))^\beta,
$$
$$
b_n=\esssup_{E_{4n}\setminus E_{2n}} \phi'(d)\Gamma(d,d)\leq MC(2-\beta)(4n+1)\log(4n+1)
$$
and
$$
c_n(g)\leq\esssup_{E_{4n}\setminus K} (\phi'(d))^2 \Gamma(d,d)
\leq MC^2 (2-\beta)^2 \left(\log(4n+1) \right)^{2-\beta}.
$$
Subsequently,
\begin{eqnarray*}
&&e^{-\phi(2n)+c_n(g) T}\mu(E_{4n}\setminus E_{2n})^{1/2}\\
&\leq& \exp\left (-C(\log(2n+1))^{2-\beta}+N(\log(n+1))^{2-\beta}+TMC^2(2-\beta)^2(\log(4n+1))^{2-\beta}\right ).
\end{eqnarray*}
Let $C:=3N$ and $T:=\frac{1}{9MN(2-\beta)^2}>0$, then we obtain
\begin{equation}\label{slim}
\lim_{n\to \infty}\frac{e^{-\phi(2n)+c_n(g) T}}{n} \widehat{A}_n(\phi)=\lim_{n\to \infty}\frac{e^{-\phi(2n)+c_n(g) T}}{n}\mu(E_{4n}\setminus E_{2n})^{1/2}(\sqrt{a_n}+b_n)=0.
\end{equation}
Consequently, by the same arguments in Corollary \ref{g1}, $\tt$ is conservative when $0\leq \beta<2$.\\
Let $\beta=2$. Define
$$
\phi(r):=3N\log(\log(r+1)+1).
$$
Then by similar calculations, we can choose $T>0$ such that (\ref{slim}) holds.\\
(ii) Choosing $\phi(r):=3Nr^2$ the proof is similar to the one of (i).
\end{proof}

\begin{exam}\label{symdf}
(cf. \cite[Section 3.1]{OU15}) Let $(\E, C_0^\infty(\R^d))$ be a symmetric bilinear form in $L^2(\R^d,dx)$ defined by
$$
\E(f,g):=\int_{\R^d} \langle A \nabla f ,\nabla g \rangle dx,
$$
where $A=(a_{ij})=(a_{ji})\in L^1_{loc}(\R^d,dx)\cap L^\infty_{loc}(K_0^c,dx)$, $1\leq i,j\leq d$ for some compact subset $K_0$ in $\R^d$. Assume that for any compact set $K$, there exists a constant $\nu_K>0$ such that
$$
\nu_K |\xi|^2 \leq \langle A(x) \xi,\xi\rangle
$$
for all $\xi \in \R^d$, $\mu$-a.e. $x\in K$. Here $\langle\ ,\ \rangle$ denotes the Euclidean inner product on $\R^d$ with corresponding norm $|\cdot |$ and $C_0^\infty(\R^d)$ is the set of infinitely often differentiable functions with compact support in $\R^d$. Then $(\E,C_0^\infty(\R^d))$ is closable and its closure ($\E,\F)$ satisfies (H1)-(H4) with $p=2$, $N\equiv 0$ and $\rho(x)=|x|$. Indeed, for each relatively compact open subset $V$ of $\R^d$, there exists a function $\chi_V\in C_0^\infty(\R^d)$ such that $\chi_V\equiv 1$ on $V$. Then $\rho \chi_V \in \F$ and $\rho \chi_V=\rho$ on $V$, hence $\rho\in \F_{loc}$. Consequently, by Proposition \ref{symexam}(i), $\tt$ is conservative if there exists a constant $M>0$ such that
$$
\frac{\langle A(x)x,x\rangle}{|x|^2}\leq M(|x|+1)^2\log(|x|+1)
$$
$dx$-a.e. outside some compact subset $K$ of $\R^d$ containing $K_0$.
\end{exam}

\begin{rem}\label{sturm}
(cf. \cite[Section 3.1]{OU15}) By Proposition \ref{symexam}(ii), we recover the result of \cite[Remarks p.185 (3.7)]{Stu1}. More precisely, \cite[Theorem 4]{Stu1} was devoted to determine the conservativeness for a symmetric strongly local regular Dirichlet form expressed as in (\ref{form}) in case that the topology induced by the intrinsic metric is equivalent to the original topology on $E$ and in case that the intrinsic balls are all relatively compact open in $E$ (cf. \cite[Assumption (A)]{Stu1}). Then by \cite[Lemma 1]{Stu1}, $\rho(\cdot):= d^{int}(\cdot , x_0)\in \F_{loc}\cap C(E)$ for any $x_0\in E$ where $d^{int}$ is the intrinsic metric and $\rho$ satisfies
$$
\Gamma (\rho,\rho)\leq 1.
$$
Applying these assumptions to our situation implies $a_n\leq 1$ for any $n\geq 1$. Hence (\ref{form}), (\ref{dloc}), (\ref{relcpt}) and (\ref{sym1}) are satisfied and thus by Proposition \ref{symexam}(ii), $\tt$ is conservative if there exists a constant $N>0$ such that $\mu(E_{4n}\setminus E_{2n})\leq \exp(2N n^2)$ for $n\gg 1$.
\end{rem}
\subsection{Sectorial perturbations of symmetric Dirichlet forms}\label{exam10}
In this subsection, we apply Theorem \ref{thm6} to non-symmetric Dirichlet forms which are divergence free perturbations of symmetric Dirichlet forms on $\R^d$.\\ \\
Let $E=\R^d$ and $d\mu=\varphi dx$ where $\varphi \in L^1_{loc}(\R^d,dx)$, $\varphi>0$ $dx$-a.e. Consider $A=(a_{ij})=(a_{ji}) \in L_{loc}^1(\R^d,\mu)\cap L^\infty_{loc}(K_0^c,\mu)$, ${1\leq i,j \leq d}$ for some compact subset $K_0$ in $\R^d$ and suppose for any compact set $K \subset \R^d$, there exists $\nu_K >0$ such that
\begin{equation}\label{lb}
\nu_K |\xi|^2 \leq \langle A(x) \xi,\xi\rangle
\end{equation}
for all $\xi \in \R^d$, $\mu$-a.e. $x\in K$. We assume that the symmetric bilinear form
$$
\E^0(f,g):=\int_{\R^d} \langle A(x) \nabla f(x),\nabla g(x)\rangle \mu(dx),\ f,g \in C_0^\infty(\R^d)
$$
is closable on $L^2(\R^d,\mu)$. Then its closure $(\E^0,D(\E^0))$ is a symmetric strongly local regular Dirichlet form. We further assume that $B=(B_1,...,B_d)\in L^2_{loc}(\R^d,\R^d,\mu)$  satisfies $|B|\in L^\infty_{loc}(K_0^c,\mu)$ and
\begin{equation}\label{div0}
\int_{\R^d} \langle B(x),\nabla f(x) \rangle \mu(dx) =0
\end{equation}
for any $f\in C_0^\infty(\R^d)$ and there exists a constant $C>0$ which is independent of $f$ and $g$ such that
\begin{equation}\label{wsc}
\left | \int_{\R^d} \langle B,\nabla f \rangle g d\mu \right | \leq C \E_1^0(f,f)^{1/2}\E_1^0(g,g)^{1/2},
\end{equation}
for any $f,g \in C_0^\infty(\R^d)$. Consider the non-symmetric bilinear form
$$
\E(f,g):=\int \langle A(x) \nabla f(x),\nabla g(x)\rangle \mu(dx)-\int \langle B(x), \nabla f(x) \rangle g(x) \mu(dx),\ f,g \in C_0^\infty(\R^d).
$$
Then $(\E,C_0^\infty(\R^d))$ is closable on $L^2(\R^d,\mu)$ and by (\ref{div0}) and (\ref{wsc}), the closure $(\E,\F)$ is a non-symmetric Dirichlet form in the sense of \cite[I. Definition 4.5]{MR}. By (\ref{lb}), (\ref{div0}) and (\ref{wsc}), we obtain
$$
\int_{\R^d} \langle B, \nabla v\rangle d\mu=0, \ \text{for any }v\in \F_{b}.
$$
Let $V_n=\{ z: |z|<4n\}$. As in Remark \ref{rem3}, we may define the part Dirichlet forms $(\E^n,\F^n)$ corresponding to the increasing sequence of relatively compact open sets  $(V_n)_{n\geq 1}$ where $\F^n=\{ u\in \F : \widetilde{u}=0$ q.e. on $ V_n^c\}$ (see \cite[Section 3.5]{O13}). Denote the coform of $(\E^n,\F^n)$ by $(\widehat\E^n, \F^n)$ and the associated semigroups of $(\widehat{\E}^n,\F^n)$ by $(\widehat{T}_t^n)_{t>0}$ and the associated linear operators by $(\widehat{L}^n,D(\widehat{L}^n))$ on $L^2(V_n,\mu)$.  Then the coform $(\widehat\E^n, \F^n)$ is also a non-symmetric Dirichlet form in $L^2(V_n,\mu)$ and
$$
\widehat{T}_t f=\lim_{n \to \infty}\widehat{T}_t^n f \ \mu\text{-a.e.}
$$
for any $f\in L^2(\R^d,\mu)$ where $\widehat{T}_t^n f:= \widehat{T}_t^n(f\cdot 1_{V_n})$. In particular, if $f$ is non-negative, then $\widehat{T}_t^n f \nearrow \widehat{T}_t f$ $\mu$-a.e. as $n\to \infty$. $(\E,\F)$ satisfies (H1) with $p=2$. Furthermore, for $v\in D(\widehat{L}^n)_b$
$$
(-\widehat{L}^n v,u)=\E(u,v)=\E^0(u,v)+\int_{\R^d} \langle B,\nabla v\rangle u d\mu \quad\text{for any }u\in \F_b.
$$
Putting $D(N)=\F_{loc,b}$ and $Nv=\langle B,\nabla v \rangle$  imply that (\ref{euv}) and (H2) hold. Choose $\rho(x):=|x|$. Then in the same way as in Example \ref{symdf}, we find that $\rho \in \F_{loc}$ and by the assumptions on $A$ and $B$, we obtain
$$
\langle A \nabla \rho,\nabla \rho\rangle, \ \langle B,\nabla \rho \rangle \in L^\infty_{loc}(K_0^c,\mu),
$$
hence (H3) holds. By \cite[I. Corollary 2.21]{MR}, ($\widehat{T}_t^n)_{t>0}$ is analytic on $L^2(V_n,\mu)$, hence (H4) holds (i.e. $D_0=L^\infty(\R^d,\mu)\cap L^2(\R^d,\mu)_0$).\\
\centerline{}
\subsubsection{Example}\label{ShTrex}
Consider the non-symmetric Dirichlet form introduced in \cite[Section 5]{RST}. There $\varphi$ is a Muckenhoupt $\A_\beta$-weight, $1\leq \beta\leq  2$ with $\varphi=\xi^2$, $\xi\in H_{loc}^{1,2}(\R^d,dx)$, $\varphi>0$ $dx$-a.e. and
$$
\frac{|\nabla \varphi |}{\varphi} \in L^p_{loc}(\R^d,dx)
$$
where $p=(d+\e)\vee 2$ for some $\e>0$, $H^{1,2}(\R^d,dx)$ is the usual Sobolev space of order one in $L^2(\R^d,dx)$ and $H_{loc}^{1,2}(\R^d,dx):=\{f: f\cdot \chi\in H^{1,2}(\R^d,dx)\text{ for any } \chi \in C_0^\infty(\R^d) \}$. Thus the symmetric bilinear form
$$
\E^0(f,g)=\int_{\R^d} \langle \nabla f(x),\nabla g(x)\rangle \mu(dx),\ f,g \in C_0^\infty(\R^d)
$$
is closable on $L^2(\R^d,\mu)$. Moreover in \cite[Section 5]{RST} it is assumed that $ |B| \in L^N_{loc}(\R^d,\mu)\cap L^\infty(K_0^c,\mu)$ for some compact set $K_0$ and some constant $N\geq \beta d +\log_2 A$, where the constant $A$ is the $\A_\beta$ constant of $\varphi$. Then by \cite[Section 5]{RST}, (\ref{wsc}) holds. The corresponding closure $(\E,\F)$ satisfies (H1)-(H4) with $D(N)=\F_{loc,b}$, $Nv=\langle B,\nabla v \rangle$ and $\rho(x)=|x|$ as in Example \ref{symdf} and $D_0=L^\infty(\R^d,\mu)\cap L^2(\R^d,\mu)_0$. In this situation, $\Gamma(\rho,\rho)=1$ and
$$
\left | \frac{\langle B ,\nabla \rho \rangle}{\rho}\right | \leq \|B \|_{L^\infty(K_0^c,\mu)}
$$
$\mu$-a.e. on $K^c$ where $K$ is an arbitrary compact subset of $\R^d$ containing $K_0$ and $\{x\in \R^d: |x|\leq 1\}$. Furthermore, since $\varphi\in \A_\beta$, we get by \cite[Proposition 1.2.7]{Tu} that there exists a constant $N>0$ such that
$$
\mu(E_{4n})\leq N r^{\beta d}.
$$
Thus, for $\phi(r):=\frac{r^2}{2}$ we obtain (cf.  (\ref{ahn})) for $n\gg 1$
$$
\widehat{A}_n(\phi) \leq \left (1+4n 
+\|B \|_{L^\infty(K_0^c,\mu)}\right )N r^{\beta d}.
$$
Consequently, $\tt$ is conservative by Corollary \ref{g1}(iii) and we recover the result of \cite[Lemma 5.4]{RST}.

\subsection{Sectorial perturbations of sectorial Dirichlet forms}\label{exam11}
In this subsection, we show that Theorem \ref{thm6} is also applicable to non-symmetric Dirichlet forms with non-symmetric diffusion matrix. The key observation is that the anti-symmetric part of the diffusion matrix becomes a $\mu$-divergence free vector field after integration by parts.\\ \\
Let $E=\R^d$ and $d\mu=\varphi^2 dx$, $\varphi\in H_{loc}^{1,2}(\R^d,dx)$, $\varphi>0$ $dx$-a.e. Let $H^{1,2}(\R^d,\mu)$ be the closure of $C_0^\infty(\R^d)$ in $L^2(\R^d,\mu)$ with respect to the norm $\left (\int_{\R^d}(|\nabla f|^2+f^2)d\mu\right )^{1/2}$ and
$$
H_{loc}^{1,2}(\R^d,\mu):=\{f: f\cdot \chi\in H^{1,2}(\R^d,\mu)\text{ for any } \chi \in C_0^\infty(\R^d) \}.
$$
Consider $A=(a_{ij})\in L^{1}_{loc}(\R^d,\mu)$, $1\leq i,j\leq d$ with symmetric part $\widetilde{A}=(\widetilde{a}_{ij})$, where $\widetilde{a}_{ij}:=\frac{1}{2}(a_{ij}+a_{ji})\in L^\infty_{loc}(K_0^c,\mu)$ for some compact subset $K_0$ in $\R^d$ and anti-symmetric part $\check{A}=(\check{a}_{ij})$, where $\check{a}_{ij}:=\frac{1}{2}(a_{ij}-a_{ji})\in H_{loc}^{1,2}(\R^d,\mu)\cap L_{loc}^\infty(\R^d,\mu)$. Suppose for any compact set $K \subset \R^d$, there exist $\nu_K >0$ and $L>0$, such that
$$
\max_{1\leq i,j\leq d} \esssup_{K} | \check{a}_{ij}| \leq L \cdot \nu_K\ \text{ and }\ \nu_K |\xi|^2 \leq \langle \widetilde{A}(x) \xi,\xi\rangle
$$
for all $\xi \in \R^d$, $\mu$-a.e. $x\in K$. Assume that $B=(B_1,...,B_d)\in L^2_{loc}(\R^d,\R^d,\mu)$ satisfies\begin{equation}\label{div1}
\int_{\R^d} \langle B, \nabla f\rangle d\mu=0,
\end{equation}
for any $f\in C_0^\infty(\R^d)$. Assume further that there exists a constant $C>0$ such that
$$
\left | \int_{\R^d} \langle B,\nabla f\rangle d\mu \right |\leq C \E_1^{\widetilde{A}}(f,f)^{1/2}\E_1^{\widetilde{A}}(g,g)^{1/2}
$$
for any $f,g\in C_0^\infty(\R^d)$, where $\E^{\widetilde{A}}(f,g):=\int_{\R^d} \langle \widetilde{A}\nabla f, \nabla g \rangle d\mu$. Likewise define $\E^{\check{A}}(f,g)$ and $\E^A(f,g)$ and set
$$
\E^{A,B}(f,g):=\E^A(f,g)-\int_{\R^d} \langle B,\nabla f\rangle g d\mu=\int_{\R^d} \langle A\nabla f,\nabla g\rangle d\mu -\int_{\R^d} \langle B,\nabla f\rangle g d\mu, \quad f,g \in C_0^\infty(\R^d).
$$
Then $(\E^{A,B},C_0^\infty(\R^d))$ is closable on $L^2(\R^d,\mu)$ and its closure $(\E^{A,B},\F)$ is a non-symmetric sectorial regular Dirichlet form.
Let $\tt$ (resp. $\tht$) be the $C_0$-semigroup of contractions on $L^2(\R^d,\mu)$ associated with $(\E^{A,B},\F)$, and $(L,D(L))$ (resp. $(\widehat{L},D(\widehat{L}))$ be the corresponding linear operator (resp. co-operator). For $f,g \in C_0^\infty (\R^d)$, we obtain by integration by parts
\begin{align}\label{betai}
\E^{\check{A}}(f,g)-\int_{\R^d} \langle B,\nabla f \rangle g d\mu &=-\sum_{i,j}^d \int_{\R^d} \left [ \check{a}_{ij}\partial_i \partial_j f + \left( \partial_j \check{a}_{ij}+\check{a}_{ij}\frac{2\partial_j \varphi}{\varphi} +B_i \right)\partial_i f \right ]g \varphi^2 dx \nonumber\\
&= -\sum_{i=1}^d \int_{\R^d}  \underbrace{\left [\sum_{j=1}^d  \left( \partial_j \check{a}_{ij}+\check{a}_{ij}\frac{2\partial_j \varphi}{\varphi} +B_i \right)\right ]}_{=:\beta_i}\partial_i f g \varphi^2 dx
\end{align}
where $\beta=(\beta_1,...,\beta_d)\in L^2_{loc}(\R^d,\R^d,\mu)$ is again a $\mu$-divergence free vector field. Indeed, for $f\in C_0^\infty(\R^d)$, we get by (\ref{div1})
$$
\int_{\R^d} \langle \beta, \nabla f \rangle d\mu=\sum_{i=1}^d\int_{\R^d}  \left [\sum_{j=1}^d  \left( \partial_j \check{a}_{ij}+\check{a}_{ij}\frac{2\partial_j \varphi}{\varphi} \right)\right ]\partial_i f \varphi^2 dx=-\int_{\R^d} \sum_{i,j}^d \check{a}_{ij}\partial_i \partial_jf \varphi^2 dx=0.
$$
Moreover, by (\ref{betai}) and since $\E^{\check{A}}$ satisfies the strong sector condition, there is a constant $C>0$ such that
$$
\left | \int_{\R^d} \langle \beta-B, \nabla f \rangle g d\mu \right |=\left |\E^{\check{A}}(f,g) \right | \leq C \E_1^{\widetilde{A}}(f,f)^{1/2}\E_1^{\widetilde{A}}(g,g)^{1/2}\ \text{for any }f,g\in C_0^\infty(\R^d),
$$
hence $\left | \int_{\R^d} \langle \beta, \nabla f \rangle g d\mu \right | \leq C \E_1^{\widetilde{A}}(f,f)^{1/2}\E_1^{\widetilde{A}}(g,g)^{1/2}$ for some constant $C>0$.
It follows that $B$ and $\beta$ satisfy the same assumptions and that
$$
\E^{A,B}(f,g)=\E^{\widetilde{A}}(f,g)-\int_{\R^d} \langle \beta,\nabla f \rangle g d\mu=:\E^{\widetilde{A},\beta}(f,g)
$$
for any $f,g\in C_0^\infty(\R^d)$. Therefore, the closures of $(\E^{A,B},C_0^\infty(\R^d))$ and $(\E^{\widetilde{A},\beta},C_0^\infty(\R^d))$ are identical and define the same Dirichlet form. We now assume that $|\beta| \in L^\infty_{loc}(K_0^c,\mu)$.\\
Let $V_n=E_{4n}=\{z: |z|<4n\}$. Then $(V_n)_{n\geq 1}$ is a sequence of relatively compact open sets. As in subsection \ref{exam10}, let $(\E^n,\F^n)$ be the part Dirichlet forms on $L^2(V_n,\mu)$ of $(\E^{\widetilde{A},\beta},\F)$ (see \cite[Section 3.5]{O13}). Let $(\widehat\E^n, \F^n)$ be the coform of $(\E^n,\F^n)$, $(\widehat{T}_t^n)_{t>0}$ be the associated semigroups of $(\widehat{\E}^n,\F^n)$ and $(\widehat{L}^n,D(\widehat{L}^n))$ be the associated linear operators on $L^2(V_n,\mu)$. Then
$$
\widehat{T}_t f=\lim_{n \to \infty}\widehat{T}_t^n f \ \mu\text{-a.e.}
$$
for any $f\in L^2(\R^d,\mu)$ where $\widehat{T}_t^n f:= \widehat{T}_t^n(f\cdot 1_{V_n})$. In particular, if $f$ is non-negative, then $\widehat{T}_t^n f \nearrow \widehat{T}_t f$ $\mu$-a.e. as $n\to \infty$. $(\E^{\widetilde{A},\beta},\F)$ satisfies (H1) with $p=2$. Furthermore, for $v\in D(\widehat{L}^n)_b$,
$$
(-\widehat{L}^n v,u)=\E^{\widetilde{A},-\beta}(u,v)=\E^{\widetilde{A}}(u,v)+\int_{\R^d} \langle \beta,\nabla v\rangle u d\mu, \quad\text{for any }u\in \F_b.
$$
Putting $D(N)=\F_{loc,b}$ and $Nv=\langle \beta,\nabla v \rangle$  imply that (\ref{euv}) and (H2) hold with $(\E^0,D(\E^0))=(\E^{\widetilde{A}},\F)$. Let $\rho(x):=|x|$ then $\rho\in \F_{loc}$ as in Example \ref{symdf}. We further obtain by the assumptions on $\widetilde{A}$ and $\beta$, that
$$
\langle \widetilde{A} \nabla \rho,\nabla \rho\rangle,\ \langle \beta,\nabla \rho \rangle \in L^\infty_{loc}(K_0^c,\mu).
$$
Hence (H3) holds. Since $(\widehat{\E}^n,\F^n)$ satisfies the weak sector condition for each $n\geq 1$, $(\widehat{T}_t^n)_{t>0}$ are analytic, i.e. (H4) holds. Consequently, by Corollary \ref{g1}(i) with $\phi(r):=C\log(r+1)$, $\rho(x)=|x|$, if there are constants $M, C>0$, and $0<\alpha<1$ such that
\begin{equation}\label{esti}
\left | \frac{\left \langle \widetilde{A}(x)x,  x \right \rangle}{|x|^2} + \frac{(|x|+1)}{C|x|}\langle \beta(x), x \rangle \right | \leq M(|x|+1)^2 \log(|x|+1),
\end{equation}
$dx$-a.e. outside some compact subset $K$ of $\R^d$ with $K\supset K_0$ and
\begin{equation}\label{ahnp}
\widehat{A}_n(\phi)\leq n (n+1)^{\alpha C}
\end{equation}
for $n\gg 1$, then $\tt$ is conservative.

\subsubsection{Example}\label{3.3.1}
The sufficient criteria (\ref{esti}) and (\ref{ahnp}) for conservativeness extend the result of \cite{TaTr} in the sense that we can also consider invariant measures $\mu=\varphi^2 dx$ where $\varphi \nequiv 1$. In this example, we show that we can also recover the result of \cite{TaTr} to some extend in case $\varphi\equiv1$.\\
Let $d\geq 3$ and $\varphi^2\equiv 1$, i.e. $\mu$ is the Lebesgue measure. Assume further that for $B=(B_1,...,B_d)\in L^d_{loc}(\R^d,\R^d,dx)$, there exist constants $L_i>0$, such that for $1\leq i\leq  d$
$$
\min \{ \|B_i^2 \|_{L^\infty(E_n)}, \| B_i\|_{L^d(E_n)} \} \leq L_i\nu_{\overline{E}_n}.
$$
Then by \cite[Section 2]{TaTr}, $(\E^{A,B},C_0^\infty(\R^d))$ is closable on $L^2(\R^d,dx)$ and the closure $(\E^{A,B},\F)$ satisfies the weak sector condition. Thus, we are able to apply (\ref{esti}) and (\ref{ahnp}) to $(\E^{A,B},\F)$ in order to determine the conservativeness. For instance, if there exists a constant $M_0>1$ such that
\begin{equation}\label{betap}
\frac{\langle \widetilde{A}(x)x,x\rangle}{|x|^2}+|\langle \beta(x),x\rangle | \leq M_0(|x|+1)^2\log(|x|+1)
\end{equation}
$\mu$-a.e. outside some compact subset $K$ of $\R^d$ with $K\supset K_0\cup \{x: |x|\leq 1\}$, then ($\E^{A,B},\F)$ is conservative. Indeed, by (\ref{betap})
\begin{equation}\label{atilde}
\left | \frac{\langle \widetilde{A}(x)x,x\rangle}{|x|^2}+\frac{(|x|+1)}{C|x|}\langle \beta(x),x\rangle \right |\leq M_0\left (1+\frac{2}{C}\right )(|x|+1)^2\log(|x|+1)
\end{equation}
and
\begin{equation}\label{betax}
\left | \frac{\langle \beta(x),x\rangle}{|x|}\right | \leq M_0 (|x|+1)^2
\end{equation}
$\mu$-a.e. on $K^c$. Let $\phi(r):= C\log(r+1)$ where the constant $C>0$ will be chosen later. It follows from (\ref{betap}), (\ref{atilde}) and (\ref{betax}) that
$$
\widehat{A}_n(\phi)=(\sqrt{a_n}+b_n)\mu(E_{4n}\setminus E_{2n})^{1/2}+\left ( \int_{E_{4n}\setminus E_{2n}} \left | \frac{\langle \beta(x),x\rangle}{|x|}\right |^2dx \right )^{1/2}\leq M' n^{d/2+2}
$$
for some constant $M'>0$. Consequently, putting $C=\frac{d}{2}+3$, $M=M_0(1+\frac{2}{C})$ and $\alpha=\frac{C-1}{C}$ implies there are constants $M, C>0$, and $0<\alpha<1$ such that (\ref{esti}) and (\ref{ahnp}) hold and $\tt$ is conservative.

\centerline{}
\begin{rem}\label{tatrrem}
Compared with the estimate \cite[p. 422]{TaTr}, (\ref{betap}) is a slightly stronger condition. Our aim was to demonstrate how quickly Corollary \ref{g1} can lead to acceptable results. Later, by applying Corollary \ref{g1} more consciously we will see that  $\left | \frac{\left \langle \widetilde{A}(x)x,  x \right \rangle}{|x|^2} \right |$ in (\ref{esti}) is allowed to have a large growth if $\langle \beta(x),x \rangle$ can compensate it (see Examples \ref{ex16} and \ref{ex1} below).
\end{rem}

\section{Non-sectorial applications on Euclidean space}\label{4}
In this section, we consider non-sectorial perturbations of symmetric Dirichlet forms on Euclidean space as introduced in \cite{GiTr}. For the convenience of the reader, we explain in concise form the construction of the underlying generalized Dirichlet form from \cite{GiTr}, how the constructed generalized Dirichlet fits into the frame of section \ref{2}, as well as some of its main properties. Subsequently, we apply the conservativeness criterion of section \ref{2} to the concrete situation and present explicit examples.
\centerline{}
\subsection{The construction scheme}\label{4.1}
Let $E\subset \R^d$ be either open or closed. If $E$ is closed, we assume $dx(\partial E)=0$ where $E$ is the disjoint union of its interior $E^0$ and its boundary $\partial E$. Let $\varphi \in L_{loc}^1 (E,dx)$ with $\varphi >0$ $dx$-a.e. and $d\mu :=\varphi dx$. Then $\mu$ is a $\sigma$-finite measure on $\B(E)$ and has full support. Let $C_0^\infty(E)$ be the set of infinitely often differentiable functions with compact support in $E$ if $E$ is open and $C_0^\infty(E):=\{ u\in E \longrightarrow \R : \exists v \in C_0^\infty(\R^d)$ with $v=u$ on $E \}$ if $E$ is closed.
\centerline{}
Consider $A=(a_{ij})=(a_{ji}) \in L_{loc}^1(E,\mu)$, ${1\leq i,j \leq d}$ and suppose for each relatively compact open set $V \subset E$, there exists $\nu_V >0$ such that
\begin{equation}\label{loc}
\nu_V^{-1} |\xi|^2 \leq \langle A(x) \xi,\xi\rangle \leq \nu_V |\xi|^2
\end{equation}
for all $\xi \in \R^d$, $\mu$-a.e. $x\in V$. We assume that
$$
\E^0(f,g):=\int_E \langle A(x) \nabla f(x),\nabla g(x)\rangle \mu(dx),\ f,g \in C_0^\infty(E)
$$
is closable on $L^2(E,\mu)$. Denote the closure of $(\E^0,C_0^\infty(E))$ on $L^2(E,\mu)$ by $(\E^0,D(\E^0))$. Then $(\E^0,D(\E^0))$ is a symmetric regular Dirichlet form on $L^2(E,\mu)$. Let $(L^0,D(L^0))$ be the linear operator corresponding to $(\E^0,D(\E^0))$ on $L^2(E,\mu)$ and $(T^0_t)_{t> 0}$ be the $C_0$-semigroup corresponding to $(L^0,D(L^0))$. \\
Let $B:=(B_1,\ldots ,B_d)\in L_{loc}^2 (E,\R^d,\mu)$ satisfy
$$
\int_E \langle B(x), \nabla f(x)\rangle \mu(dx)=0
$$
for any $f\in C_0^\infty(E)$. \\
The following construction from \cite{GiTr} works for any increasing sequence of relatively compact open sets $(V_n)_{n\geq 1}$ in $E$ such that $\overline{V}_n\subset V_{n+1}$, $n\geq 1$, and $\cup_{n\geq 1}V_n=E$. Since we need to assume (H3) later and want to simplify notations we assume from now on that
\begin{itemize}
\item[(B)] there exists a non-negative continuous function $\rho \in D(\E^0)_{loc}$ such that
$$
E_n:=\{x\in E:  \rho(x)< n\}
$$
is a relatively compact open set in $E$ and $\cup_{n\geq 1} E_n =E$ and $\langle B, \nabla \rho \rangle \in L^\infty_{loc}(K_0^c,\mu)$ for some compact subset $K_0$ in $E$.
\end{itemize}
Let
$$
V_n:=E_{4n}, \ n\ge 1.
$$
Then $(V_n)_{n \geq 1}$ is an increasing sequence of relatively compact open sets in $E$ such that $\overline{V}_n\subset V_{n+1}$ and $\cup_{n\geq 1}V_n=E$. Let $C_0^\infty(V_n):=\{u\in \ci : $ supp$(u)\subset V_n\}$ and $(\E^{0,n},D(\E^{0,n}))$ be the symmetric Dirichlet form on $L^2(V_n,\mu)$ given as the closure of
$$
\E^{0,n}(f,g):=\int_{V_n} \langle A(x) \nabla f(x),\nabla g(x)\rangle \mu(dx),\ f,g \in C_0^\infty(V_n).
$$
Let $(L^{0,n},D(L^{0,n}))$ be the closed linear operator on $L^2(V_n,\mu)$ associated with $(\E^{0,n},D(\E^{0,n}))$.
Then, by \cite[Section 4]{GiTr} (cf. also \cite[Proposition 1.1]{St1}), there exists a closed linear operator $(\overline{L}^n,D(\overline{L}^n))$ on $L^1(V_n,\mu)$ which is the closure of
$$
L^{0,n}u +\langle B,\nabla u \rangle, u \in D(L^{0,n})_b
$$
on $L^1(V_n,\mu)$ and which generates a sub-Markovian $C_0$-semigroup of contractions on $L^1(V_n,\mu)$. Let $(L^n,D(L^n))$ be the part of $(\overline{L}^n,D(\overline{L}^n))$ on $L^2(V_n,\mu)$ and $(T_t^n)_{t>0}$ be its sub-Markovian $C_0$-semigroup on $L^2(V_n,\mu)$.
Then $D(L^n)\subset D(\E^{0,n})$ can be seen as in \cite[proof of Lemma 3.1]{St1}.
Proceeding in the same way as just explained, there exists a linear operator $(\widehat{L}^n,D(\widehat{L}^n))$ on $L^1(V_n,\mu)$ which is the closure of
$$
L^{0,n}v -\langle B,\nabla v \rangle,\ v\in D(L^{0,n})_b
$$
on $L^1(V_n,\mu)$ and which satisfies $D(\widehat{L}^n)_b\subset D( \E^{0,n})_b$
$$
-\int_{V_n} \widehat{L}^n vvd\mu=\E^{0,n}(v,v)
$$
and
\begin{equation}\label{h2}
-\int_{V_n} \widehat{L}^nvu d\mu= \E^{0}(u,v)+\int_{V_n} \langle B,\nabla v \rangle ud\mu
\end{equation}
for any $v\in D(\widehat{L}^n)_b$ and $u\in D(\E^{0,n})_b$. Let $(\widehat{T}_t^n)_{t>0}$ be the $C_0$-semigroup of contractions on $L^1(V_n,\mu)$ corresponding to $(\widehat{L}^n,D(\widehat{L}^n))$. Let $(G_\alpha^n)_{\alpha>0}$ (resp. $(\widehat{G}_\alpha^n)_{\alpha>0})$ be the resolvent of $(T_t^n)_{t>0}$ (resp. $(\widehat{T}_t^n)_{t>0}$) on $L^2(V_n,\mu)$ (resp. $L^1(V_n,\mu)$). \\
Define for $f\in L^2(E,\mu)$,
$$
G_\alpha^n f:= G_\alpha^n (f \cdot 1_{V_n}),\ \alpha>0.
$$
Then $(G_\alpha^n)_{\alpha>0}$, $n\geq 1$, gives rise to a sub-Markovian $C_0$-resolvent of contractions on $L^2(E,\mu)$. Indeed, let $f\in L^2(E,\mu)_b$, with $f\geq 0$ $\mu$-a.e. and $\alpha>0$. Let $w_\alpha:= G_\alpha^n f-G_\alpha^{n+1}f$. Then $w_\alpha, w_\alpha^+, w_\alpha^-\in D(\E^{0,n+1})_b\subset D(\E^{0})$  
but also $w_\alpha^+\in D(\E^{0,n})_b$ since $\lim_{k\to \infty} (v_k-G_\alpha^{n+1}f)^+=w_\alpha^+$ weakly in $D(\E^{0,n})$ for any sequence $(v_k)_{k\ge 1}\subset C_0^{\infty}(V_n)$ such that $\lim_{k\to \infty}v_k=G_\alpha^{n}f$ strongly in $D(\E^{0,n})$. Since
$$
\E^{0}(w_\alpha^+,w_\alpha^-)=\E^{0}(w_\alpha^+,w_\alpha^+-w_\alpha)=-\E^{0}((-w_\alpha)\wedge 0, (-w_\alpha)-(-w_\alpha)\wedge 0)\leq 0
$$
and
$$
\int_{V_{n+1}} \langle B, \nabla w_\alpha\rangle  w_\alpha^+ d\mu= \int_{V_{n+1}}  \langle B, \nabla w_\alpha^+\rangle w_\alpha^+ d\mu=0,
$$
using in particular the dual version of (\ref{h2}), we get 
\begin{align*}
\E_\alpha^{0}(w_\alpha^+,w_\alpha^+)&\leq \E_\alpha^{0}(w_\alpha,w_\alpha^+)-\int_{V_{n+1}}  \langle B, \nabla w_\alpha\rangle w_\alpha^+ d\mu\\
&= \int_{V_{n+1}} (\alpha-L^n)G_\alpha^n f w_\alpha^+d\mu -\int_{V_{n+1}} (\alpha-L^{n+1})G_\alpha^{n+1} f w_\alpha ^+ d\mu=0.
\end{align*}
Thus $w_\alpha^+=0$ $\mu$-a.e., i.e. $G_\alpha^nf\leq G_\alpha^{n+1}f$ $\mu$-a.e. (cf. \cite[Lemma 1.6]{St1}). Define for $f\in L^2(E,\mu)_b$, with $f\geq 0$ $\mu$-a.e.
$$
G_\alpha f:= \lim_{n\to \infty}G_\alpha^n f.
$$
Let $f\in L^2(E,\mu)$, $f\geq 0$ and $(f_n)_{n\geq 1}\subset L^2(E,\mu)_b$ with $0\leq f_n\leq f_{n+1}$ $\mu$-a.e. for every $n\geq 1$ be such that $f_n\to f$ in $L^2(E,\mu)$ as $n\to \infty$. Then
$$
G_\alpha f:=\lim_{n\to \infty} G_\alpha f_n
$$
exists $\mu$-a.e. since it is an increasing sequence. For general $f\in L^2(E,\mu)$, let $G_\alpha f:= G_\alpha f^+-G_\alpha f^-$.
By \cite{GiTr} one can see $\ga$ is a sub-Markovian $C_0$-resolvent of contractions on $L^2(E,\mu)$ provided
\begin{itemize}
\item[(C)] $D(L^0)_{0,b}$ is a dense subset of $L^1(E,\mu)$,
\end{itemize}
which we assume from now on. Let $(L,D(L))$ be the generator of $\ga$ and $\tt$ be the $C_0$-semigroup associated with $(L,D(L))$. Let $(\widehat{L},D(\widehat{L}))$ be the adjoint operator of $(L,D(L))$ and $(\widehat{T}_t)_{t>0}$ (resp. $(\widehat{G}_\alpha)_{\alpha>0})$ be the $C_0$-semigroup (resp. $C_0$-resolvent) associated with $(\widehat{L},D(\widehat{L}))$. Then, we obtain a generalized Dirichlet form $\E$ defined by
$$
\E(u,v):= \begin{cases}
\ \displaystyle(- Lu, v) & u \in{D(L)},\ v\in {L^2(E,\mu)}\\
\ \displaystyle(-\widehat{L}v,u) & u\in  L^2(E,\mu), \ v\in D(\widehat{L}),
\end{cases}
$$
satisfying $D(L)\subset D(\E^0)$,
$$
\E(u,v)=\E^0(u,v)-\int_E \langle B,\nabla u \rangle vd\mu, \ u\in D(L)_b, v\in D(\E^{0,n})_b
$$
for some $n\geq 1$ and
$$
\E^0(u,u)\leq \E(u,u), \  u\in D(L),
$$
i.e. $\A\equiv 0$ on $\V=L^2(E,\mu)$ in the beginning of section \ref{2}.

\subsection{Conservativeness}\label{4.2}
By construction $\tht$ satisfies
$$
\widehat{T}_t f=\lim_{n\to \infty} \widehat{T}_t^n(f\cdot 1_{V_n})
$$
for any $f\in L^1(E,\mu)\cap L^\infty(E,\mu)$. Thus, (H1) holds with $p=1$. Let $D(N)=D(\E^0)_{loc,b}$. Putting
$$
Nv:=\langle B,\nabla v\rangle
$$
then by (\ref{h2}), (H2) holds. By assumption (B), (H3) holds.
By construction of $(\widehat{L}^n,D(\widehat{L}^n))$,
$$
f\in D(\widehat{L}^n)\text{ whenever } f\in D(L^0)_{0,b}\text{ for }n\geq k_f,
$$
i.e. (C) implies that (A) holds. Since
$$
D(L^0)_{0,b} \subset D_0=\{ f: f\in L^\infty(E,\mu)\cap L^2(E,\mu)_0 \text{ such that } \widehat{T}_s^n f \in D(\widehat{L}^n), \text{ for any }n\geq k_f,\ s\in [0,t] \},
$$
(C) also implies (H4). Thus, under the assumptions (B) and (C), Corollary \ref{g1} applies with $\Gamma(\rho,\rho)=\langle A\nabla \rho,\nabla \rho \rangle$, $N(\rho)=\langle B,\nabla \rho\rangle$, $\rho$ as in (B). This gives the following corollary. Recall that in the present situation
\begin{eqnarray*}
\widehat{A}_n(\phi) &=& \left (\sqrt{\esssup_{E_{4n}\setminus E_{2n}} \langle A \nabla\rho,\nabla \rho\rangle }+\esssup_{E_{4n}\setminus E_{2n}}\phi'(\rho) \langle A \nabla\rho,\nabla \rho\rangle\right ) \left(\int_{E_{4n}\setminus E_{2n}}\varphi(x)dx\right)^{1/2}\\
&&+ \| \langle B,\nabla \rho \rangle \|_{L^2(E_{4n}\setminus E_{2n},\mu)}.
\end{eqnarray*}
\begin{cor}\label{cor13} Assume (B) and (C) and the basic assumptions on $\varphi$, $A$, $B$ of subsection \ref{4.1}.
\begin{itemize}
\item[(i)] Assume there are constants $M,C>0$, $0<\alpha<1$ and $0\leq \beta<2$ such that
$$
 \left |  \langle A\nabla \rho,\nabla\rho\rangle + \frac{(\rho+1)\langle B,\nabla \rho\rangle}{C(2-\beta)(\log(\rho+1))^{1-\beta}}  \right |\leq M(\rho+1)^2 (\log(\rho+1))^\beta,
$$
$\mu$-a.e. outside some arbitrary compact subset $K$ of $E$ with $K\supset K_0$ and
$$
\widehat{A}_n(\phi) \leq n \exp(\alpha C \left (\log(n+1)\right )^{2-\beta}),\ \phi(r)=C(\log(r+1))^{2-\beta},
$$
for $n\gg 1$. Then $\tt$ is conservative.
\item[(ii)]
Assume there are constants $M,C>0$ and $0<\alpha<1$
$$
 \left |  \langle A\nabla \rho,\nabla\rho\rangle + \frac{1}{C}{(\rho+1)(\log(\rho+1))\langle B,\nabla \rho\rangle}  \right |\leq M(\rho+1)^2 (\log(\rho+1))^2,
$$
$\mu$-a.e. outside some arbitrary compact subset $K$ of $E$ with $K\supset K_0$ and
$$
\widehat{A}_n(\phi) \leq n\log(n+1)^{C\alpha},\ \phi(r)=C\log(\log(r+1)+1)
$$
for $n\gg 1$. Then $\tt$ is conservative.
\item[(iii)] Assume that there are constants $M,C>0$ and $0<\alpha<2$ such that
$$
\left | \langle A\nabla \rho,\nabla\rho\rangle+ \frac{ \langle B,\nabla \rho\rangle}{C\rho} \right | \leq M
$$
$\mu$-a.e. outside some arbitrary compact subset $K$ of $E$ with $K\supset K_0$ and
$$
\widehat{A}_n(\phi)\leq n\exp(\alpha C n^2), \ \phi(r)=\frac{Cr^2}{2}
$$
for $n\gg 1$. Then $\tt$ is conservative.
\end{itemize}
\end{cor}
\centerline{}

\subsubsection{Example one}\label{ex16}
We first consider a multi-dimensional example where a large variance is compensated by a strong anti-symmetric part of the drift.\\
Let $E=\R^2$ and $d\mu=\varphi dx$, where $\varphi=\xi^2$  with $\xi \in H^{1,2}_{loc}(\R^d,dx)$, $\xi>0$ $dx$-a.e. is such that
$$
\varphi(x)=\frac{1}{5}|x|(|x|+1),\quad \mu\text{-a.e. }x\in K^c
$$
where $K$ is a compact subset of $\R^2$. Assume that $A=(a_{ij})=(a_{ji})\in H^{1,2}_{loc}(\R^2,\mu)$, $1\leq i,j \leq 2$ is  locally strictly elliptic (see (\ref{loc})). Then the symmetric bilinear form
$$
\E^0(f,g):=\int_{\R^2} \langle A(x) \nabla f(x),\nabla g(x)\rangle \mu(dx),\ f,g \in C_0^\infty(\R^2)
$$
is closable on $L^2(\R^2,\mu)$ by \cite[I. Proposition 3.3]{MR}. We further assume
$$
|a_{11}(x)|,\ |a_{12}(x)| \leq M_0 (|x|+1)^2\log(|x|+1),
$$
$\mu$-a.e. $x\in K^c$ for some constant $M_0>0$ and
$$
a_{22}(x)=\frac{x_1^4}{x_2},\quad  \mu\text{-a.e. }x\in K^c
$$
where $x=(x_1,x_2)\in \R^2$. Because of the large $a_{22}(x)$, $(\E^0,D(\E^0))$ does not satisfy (\ref{first}). Let $B(x):=\frac{1}{\varphi(x)}(x_2^2,-x_1^4)$. Then $B\in L^2_{loc}(\R^2,\R^2,\mu)$ and $|B|\in L^\infty_{loc}(K^c,\mu)$. Since $B(x)=\frac{1}{\varphi}(\partial_2h, -\partial_1h)$, where $h(x)=\frac{1}{5}x_1^5+\frac{1}{3}x_2^3 \in C^\infty (\R^2)$,
$$
\int_{\R^2} \langle B,\nabla f\rangle d\mu=\int_{\R^2} (\partial_2h \partial_1 f - \partial_1h \partial_2 f) dx=0,\ \text{for any } f\in C_0^\infty(\R^2).
$$
Then by the construction scheme of \ref{4.1}, we obtain a generalized Dirichlet form $\E$ given as an extension of
$$
\int_{\R^2} \langle A(x) \nabla f(x),\nabla g(x)\rangle \mu(dx)-\int_{\R^2} \langle B(x), \nabla f(x) \rangle g(x) \mu(dx),\ f,g \in C_0^\infty(\R^2)\cap D(L^0)=C_0^\infty(\R^2).
$$
As we have seen in subsection \ref{4.1}, (H1) and (H2) hold with $p=1$, $D(N)=D(\E^0)_{loc,b}$ and $Nv=\langle B, \nabla v\rangle$. Let $\rho(x):=|x|$. Then $\rho\in D(\E^0)_{loc}$ and we obtain
$$
\langle A \nabla \rho,\nabla \rho\rangle, \ \langle B,\nabla \rho \rangle \in L^\infty_{loc}(K^c,\mu),
$$
hence (H3) holds. Since $C_0^\infty(\R^2)\subset D(L^0)_{0,b}$, (C) holds (i.e. (H4) holds). For $x\in K^c$, it holds
\begin{align*}
\left | \langle A(x) \nabla |x|, \nabla |x| \rangle + \frac{(|x|+1)}{5}\langle B(x),\nabla |x| \rangle \right |&= \frac{1}{|x|^2}\left | a_{11}(x) x_1^2+2a_{12}(x)x_1x_2+a_{22}(x)x_2^2+x_1x_2^2-x_1^4x_2 \right |\\
&= \frac{1}{|x|^2}\left | a_{11}(x) x_1^2+2a_{12}(x)x_1x_2+x_1x_2^2 \right |\\
&\leq M(|x|+1)^2\log(|x|+1)
\end{align*}
for some constant $M>0$. Let
$$
\phi(r):=5\log(r+1),
$$
i.e. $C=5$ and $\beta=1$ in Corollary \ref{cor13}. Then we obtain for $n\gg 1$, and some positive constant $N$
$$
a_n\leq N n^4, \ b_n\leq N n^3, \ \| \langle B,\nabla \rho\rangle \|_{L^2(E_{4n}\setminus E_{2n},\mu)} \leq N n^4,
$$
which implies
$$
\widehat{A}_n(\phi)\leq N n^4.
$$
Now choose $\alpha=\frac{4}{5}$ in Corollary \ref{cor13} and obtain that $\tt$ is conservative.

\subsubsection{Example two}\label{ex1}
Let $d=1$ and $d\mu=\varphi dx$ where
$$
\varphi(x):= \begin{cases} \ 1 \quad &\text{if}\quad x>-1,\\
\ \displaystyle \frac{1}{|x|^3}   &\text{if}\quad x\leq-1.
\end{cases}
$$
Then $\varphi\in L_{loc}^1(\R,dx)$ and $\mu$ is a $\sigma$-finite(not finite) measure on $\B(\R)$. Let
$$
A(x):= \begin{cases} \ \displaystyle (x+\sqrt 2)^2\quad &\text{if}\quad x\geq 0,\\
\ \displaystyle\frac{x^4-x^3+6}{3}  &\text{if}\quad x<0.
\end{cases}
$$
Let $(\E^0,D(\E^0))$ be the symmetric Dirichlet form on $L^2(\R,\mu)$, which is the closure of $(\E^0,C_0^\infty(\R))$ on $L^2(\R,\mu)$ defined by
$$
\E^0(f,g):=\int_\R A(x)f'(x)g'(x)\mu(dx), \ f,g\in C_0^\infty(\R).
$$
Let $d$ be the metric induced by Euclidean norm, i.e. $d(x,y)=|x-y|$. Put $\rho(x):=d(x)=|x|$. Since $\Gamma(\rho,\rho)(x)=A(x)(\rho'(x))^2=A(x)$, the first condition (\ref{first}) in Corollary \ref{cor13} can not hold. Let $d^{int}$ be the intrinsic metric, i.e.
$$
d^{int}(x,y):=\sup\left \{u(x)-u(y): u\in D(\E^0)_{loc}\cap C(\R),\ \Gamma(u,u)\leq 1 \ \text{on } \R \right \}.
$$
Then we obtain that
$$
d^{int}(x,y)=\left |\int_x^y \frac{1}{\sqrt{A(z)} }dz \right |.
$$
Indeed, let us fix $y\in \R$, then
$$
u(x):=\left |\int_x^y \frac{1}{\sqrt{A(z)} }dz \right | \in D(\E^0)_{loc}\cap C(\R)
$$
satisfies $u(y)=0$ and $\Gamma(u,u)=A\cdot (u')^2=1.$ By definition of $d^{int}$, $d^{int}(x,y)\geq u(x)$. Suppose that $d^{int}(x,y)>u(x)$, then there exists $v\in D(\E^0)_{loc}\cap C(\R)$ such that $\Gamma(v,v)\leq 1$ and $v(x)-v(y)> u(x)-u(y)$. However $\Gamma(v,v)\leq 1$ implies that
$$
-\frac{1}{\sqrt{A}} \leq v' \leq \frac{1}{\sqrt{A}},
$$
which further implies the contradiction $v(x)-v(y)\leq u(x)-u(y)$.\\
We have $\int_{-\infty}^0 \frac{1}{\sqrt{A(z)}} dz<\infty$, so $(-\infty, 0)\subset B_R^{d^{int}}$ for some $R>0$. In other words, the ball $B_R^{d^{int}}$ induced by the metric $d^{int}$ is not a relatively compact set in $\R$. Thus assumption (A) in \cite{Stu1} does not hold and we also can not apply \cite[Theorem 4]{Stu1} to determine the conservativeness of ($\E^0,D(\E^0))$. However, by a scale function argument, we may show that $(\E^0,D(\E^0))$ is conservative. Indeed, since $A \varphi$ is continuous and strictly positive,
$$
h(x):=\int_0^x \frac{1}{A(y)\varphi(y)}dy
$$
is well-defined and satisfies
$$
\E^0(h,g)=0 \quad \text{for any }g\in C_0^\infty(\R)
$$
which implies that $h$ is harmonic, i.e. $L^0 h=0$. Thus we may regard $h$ as canonical scale and $\frac{1}{h'A}dx=\varphi dx$ as the corresponding speed measure. Define
$$
\Phi(x):=\int_0^x (h(x)-h(y)) \varphi(y)dy.
$$
Then by Feller's test for non-explosion (see for instane \cite[Chapter 3.6]{McK}), $(\E^0,D(\E^0))$ is conservative, if and only if
$$
\lim_{x\to \infty}\Phi(x)=\lim_{x\to -\infty}\Phi(x)=\infty.
$$
If $x\geq 0$, then $\varphi(x)\equiv 1$ and
$$
\Phi(x)=xh(x)-\int_0^x h(y)dy=\int_0^x h'(y)ydy=\int_0^x\frac{y}{(y+\sqrt{2})^2}dy \to \infty
$$
as $x\to \infty$. In case $x<-1$, then
$$
h(x)=c_1 +\int_{-1}^x \frac{-3y^3}{y^4-y^3+6}dy=c_1 -\int_{1}^{-x} \frac{3y^3}{y^4+y^3+6}dy
$$
for some constant $c_1$, hence $h(x)\leq -\frac{3}{8}\log(-x)+c_1$ and $\lim_{x\to -\infty} h(x)=-\infty$. Furthermore, for $x<-1$
\begin{align*}
\Phi(x)=\int_0^x (h(x)-h(y))\varphi(y)dy&=h(x)\left ( c_2+ \int_{-1}^x \varphi(y)dy \right ) +c_3- \int_{-1}^x h(y)\varphi(y)dy\\
&=h(x)\left ( c_2+ \int_{-1}^x \frac{1}{-y^3}dy \right ) +c_3- \int_{-1}^x \frac{h(y)}{-y^3}dy\\
&\geq h(x)\left (c_2 +\frac{1}{2x^2}-\frac{1}{2}\right )+c_3+\int_{-1}^x \frac{-3\log(-y)+8c_1}{8y^3}dy
\end{align*}
where $c_2<0$, $c_3>0$ are some constants. Thus, $\lim_{x\to -\infty}\Phi(x)=\infty.$ Consequently, $(\E^0,D(\E^0))$ is conservative.\\
Let $B(x):= \frac{1}{\varphi(x)}$. Then $|B|\in L^2_{loc}(\R,\mu)$ and satisfies
$$
\int_\R  B(x)  f'(x) \mu(dx)=\int_\R f'(x)dx=0
$$
for any $f\in C_0^\infty(\R)$. Consequently, by the construction scheme of subsection \ref{4.1}, we can construct a generalized Dirichlet form $\E$ given as an extension of
$$
\int_\R A(x)f'(x)g'(x)\mu(dx)-\int_\R B(x)f'(x)g(x)\mu(dx)\quad f,g\in C_0^\infty (\R) .
$$
Let $\rho(x)=|x|$. Then in the same way as in Example \ref{ex16}, we can obtain (H1)-(H4). If $x\geq 1$, then
$$
\left | A(x)+\frac{(|x|+1)}{3}\Big \langle B(x),\frac{x}{|x|} \Big \rangle \right |\leq|(x+\sqrt{2})^2+\frac{1}{3}(x+1)|
$$
and if $x\leq -1$, then
$$
\left | A(x)+\frac{(|x|+1)}{3}\Big \langle B(x),\frac{x}{|x|} \Big \rangle \right |=2.
$$
Consequently,
$$
 \left | A(x)+\frac{(|x|+1)}{3}\Big \langle B(x),\frac{x}{|x|} \Big \rangle \right | \leq M (|x|+1)^2,
$$
where $M>0$ is constant, i.e. $C=3$, $\beta=1$ and $\phi(r):=3\log(r+1)$ in Corollary \ref{cor13}. Furthermore, for $n\gg 1$,
$$
\widehat{A}_n(\phi)\leq Nn^{\frac{7}{2}}
$$
where $N>0$ is some constant depending on $T>0$. Now choose $\alpha:=\frac{5}{6}$ in Corollary \ref{cor13} and obtain that $\tt$ is conservative.

\begin{rem}\label{rem15}
Since the above example is an example for a diffusion in $\R$, we are able to symmetrize $\E$ as done in \cite[3.2.2]{GT}, i.e. there is a symmetric Dirichlet form $(\~\E,D(\~\E))$ in $L^2(\R,\~\mu)$ whose semigroup is locally equal to the semigroup $\tt$ of $\E$. Indeed, $(\~\E,D(\~\E))$ can be expressed as the following form
$$
\~\E(f,g)=\int_\R A(x)f'(x)g'(x) d\~\mu
$$
where $d\~\mu=\~\varphi dx$ and $\~\varphi(x)=\exp\left( \int_0^x \frac{\varphi'}{\varphi}(s)+\frac{B}{A}(s)ds \right)$. By the same reason as for $(\E^0,D(\E^0))$ in the example above, we can not apply \cite[Theorem 4]{Stu1} to determine the conservativeness of $\widetilde{\E}$. However, by our results on the non-symmetric realization $\E$ of $\widetilde{\E}$ we obtain that $(\widetilde{\E},D(\widetilde{\E}))$ is conservative.

\end{rem}

{\bf Acknowledgments.} This research was supported by Basic Science Research Program through the National Research Foundation of NRF-2012R1A1A2006987.

\bigskip
\noindent
{\it Gerald Trutnau\\Department of Mathematical Sciences and\\ 
Research Institute of Mathematics of Seoul National University\\
599 Gwanak-Ro, Gwanak-Gu, Seoul 08826, South Korea\\
E-mail: trutnau@snu.ac.kr\\ \\
Minjung Gim \\
National Institute for Mathematical Sciences\\ 
70 Yuseong-daero 1689 beon-gil, Yuseong-gu, Daejeon 34047, South Korea\\
E-mail: mjgim@nims.re.kr}

\end{document}